\documentclass[12pt,a4paper]{amsart}
\usepackage{amsmath,amsthm,amsfonts,graphicx,amssymb,amscd}
\usepackage{graphicx}
\usepackage{hyperref}

\newtheorem{thm}{Theorem}

\newtheorem{cor}[thm]{Corollary}
\newtheorem*{jthm}{Theorem}
\newtheorem{lem}[thm]{Lemma}
\newtheorem{prop}[thm]{Proposition}

\newtheorem{prob}[thm]{Problem}

\theoremstyle{definition}
\newtheorem{defn}[thm]{Definition}

\newcommand{\R}{\mathbb{R}}

\newcommand{\C}{\mathbb{C}}

\DeclareMathOperator{\dist}{dist}

\DeclareMathOperator{\Hopf}{Hopf}

\usepackage{color}

\begin{document}

\markright{\today}

\title[Meromorphic quadratic differentials and measured foliations]{Meromorphic quadratic differentials and measured foliations on a Riemann surface}

\author{Subhojoy Gupta}
\author{Michael Wolf}

\address{Department of Mathematics, Indian Institute of Science, Bangalore 560012, India.} 

\email{subhojoy@math.iisc.ernet.in}

\address{Department of Mathematics, Rice University, Houston, Texas, 77005-1892, USA.}
\email{mwolf@rice.edu}

\date{Compiled \today\ }
\maketitle

\begin{abstract}
We describe the space of measured foliations induced on a compact Riemann surface by meromorphic quadratic differentials. We prove that any such foliation is realized by a unique such differential $q$ if we prescribe, in addition, the principal parts of $\sqrt q$ at the poles. This generalizes a theorem of Hubbard and Masur for holomorphic quadratic differentials. The proof analyzes infinite-energy harmonic maps from the Riemann surface to $\mathbb{R}$-trees of infinite co-diameter, with prescribed behavior at the poles. 
\end{abstract}

\section{Introduction}
Let $S$ be a smooth compact oriented surface of genus $g\geq 2$, and let $\Sigma$ denote a Riemann surface structure on $S$. Holomorphic $1$-forms on $\Sigma$ are holomorphic sections of the canonical line bundle $K$ on $\Sigma$. It is a consequence of classical Hodge theory is that the space of such differentials can be identified with the first cohomology of  the surface with real coefficents:
\begin{equation*}
H^0(\Sigma, K) \cong H^1(S, \mathbb{R})
\end{equation*}
where the identification is via the imaginary parts of periods, namely,  $\omega \mapsto \Im \displaystyle\int\limits_{\gamma_i} \omega$ where $\gamma_i$ varies over a basis of homology.

The analogue of this identification for holomorphic \textit{quadratic} differentials, is the seminal theorem of Hubbard and Masur (\cite{HubbMas}). Instead of first cohomology, they consider the space $\mathcal{MF}$ of smooth objects called \textit{measured foliations} on $S$,  characterized, up to a topological equivalence, by the induced \textit{transverse measures} of simple closed curves.  Any holomorphic quadratic differential induces a measured foliation where the measure of a transverse loop $\gamma$ is $\displaystyle\int\limits_{\gamma}  \lvert  \Im \sqrt q \rvert$: these transverse measures can thus be thought of as encoding periods of the differential. 

 Their theorem asserts that this is in fact a bijective correspondence:
 
\begin{jthm}[Hubbard-Masur] Fix a compact Riemann surface $\Sigma$ of genus $g\geq 2$. Then any measured foliation $F\in \mathcal{MF}$ is realized by a unique holomorphic quadratic differential $q$ on $\Sigma$, or equivalently, its induced measured foliation is measure-equivalent to $F$, that is, induces identical measures on any simple closed curve.
\end{jthm}

  This theorem bridges the complex-analytic and topological perspectives on Teichm\"{u}ller theory.
In this paper we generalize the Hubbard-Masur theorem  to the case of meromorphic quadratic differentials with higher order poles.
This extends our work \cite{GW15} on such differentials with second order poles, completing the treatment of meromorphic quadratic differentials.
\\

A meromorphic quadratic differential $q$ with poles of higher orders (greater than two) has an induced measured foliation with ``pole-singularities", discussed in \S2.3. 
The transverse measure of an arc to a pole is then infinite, and that of a loop around the pole is determined by the residue at the pole. In \S3, the space $\mathcal{MF}(n_1,n_2,\ldots n_k)$ of such measured foliations on $S$, with $k$ singularities  of the given pole orders $n_i \geq 3$, up to bounded isotopy and Whitehead moves, is shown to be homeomorphic to $\mathbb{R}^{\chi}$ where $\chi = {6g-6 + \sum\limits_i (n_i +1)}$, where a pole of order $n_i\geq 3$ has $(n_i-2)$ local parameters determined by the foliation around it, together with the transverse measure of a loop around the pole.

On the complex-analytical side, given a choice of a coordinate chart $U\cong \mathbb{D}$ around any pole,  we can obtain the \textit{principal part}  $P(q)$, which comprises the terms with negative powers of $z$ in the expression for the meromorphic $1$-form $\sqrt q$ in these coordinates.  Moreover, we say that a principal part $P$ is \textit{compatible} with a foliation $F$ if the real part of its residue agrees with that determined by the local parameters of the foliation at each pole (this pertains to poles of even order - see Definition~\ref{compat} for details.)  \\

In this article we shall prove:

\begin{thm}\label{thm1} Let $(\Sigma, \mathcal{P})$ be a closed Riemann surface with a non-empty  set of marked points $\mathcal{P} = \{p_1,p_2,\ldots p_k\}$ such that $\chi (\Sigma \setminus \mathcal{P}) <0$. For  each $1\leq i\leq k$, fix local coordinates around $p_i$ and let $n_i\geq 3$. 

 Then given 
 \begin{itemize}
 
 \item  a measured foliation $F\in \mathcal{MF}(n_1,n_2,\ldots n_k)$ and 
 
 \item compatible  principal parts $P_i$ at each $p_i$,
 
 \end{itemize}
 there exists a unique meromorphic quadratic differential on $\Sigma$ with  a pole of order $n_i$ at $p_i$ with a principal part $P_i$, and  an induced foliation that is measure-equivalent to $F$.

\end{thm}

\textit{Remarks.} (i) We shall see in Lemma \ref{comp} that the space $\mathsf{Comp}(F)$ of principal parts compatible with any fixed foliation $F$ is homeomorphic to $\prod_{i=1}^k(\mathbb{R}^{n_i-2}\times S^1)$. On varying the foliation in $\mathcal{MF}(n_1,n_2,\ldots n_k)$ which is parametrized by $ {6g-6 + \sum\limits_i (n_i +1)}$ real parameters,  we have a total of $6g-6 + 2\sum\limits_i n_i $ real numbers as local parameters for the total space of the quadratic differentials under consideration. This matches the dimension of the space of meromorphic quadratic differentials with a poles of orders $n_1,n_2,\ldots n_k$.

(ii) One of the features of the main theorem is that it identifies the parameters responsible for the non-uniqueness of meromorphic quadratic differentials realizing a given measured foliation in terms of analytical data at the poles, namely, the coefficients of the principal parts (of a choice of  square roots in neighborhoods of the poles). A parallel of this for meromorphic $1$-forms is the classical theorem that  given a Riemann surface, such a differential is uniquely specified if one specifies the periods and the principal parts satisfying a zero-sum condition on the residues. Our main theorem can thus be considered an analogue of this fact, for meromorphic quadratic differentials.\\

As noted above, the case of poles of order two is dealt with in a separate paper (\cite{GW15}), so the statement of Theorem \ref{thm1} holds when some $n_i=2$ as well. In our previous work in \cite{GW1}, we also proved Theorem \ref{thm1} for a special foliation $F_0 \in \mathcal{MF}_k$ which has a ``half-plane" structure. There, instead of specifying the coefficients of the principal parts, we considered parameters determined by the induced singular-flat metric structure around each pole. For the general case we treat in this paper, the induced singular-flat geometry on the Riemann surface comprises more than just half-planes; it may include infinite strips and spiral domains as well. In fact, the transverse measures across the strips to the poles contribute to parameters for the measured foliations with pole-singularities (see \S3). 

In the final section (\S5), we discuss the relation with singular flat geometry. We consider the  total bundle of meromorphic quadratic differentials with the given poles over the Teichm\'{u}ller space of the punctured surface, and observe that the subspace realizing a fixed \textit{generic} foliation is locally parametrized by shearing along the strips (see Proposition \ref{shear}). Such a generic case was considered in the work of Bridgeland-Smith (\cite{BriSmi}), who relate the singular-flat geometry to stability conditions in certain abelian categories. 

 Another feature of this geometry induced by a holomorphic quadratic differential is that the flat coordinates allows one to define an $\mathrm{SL}_2(\mathbb{R})$-action on the bundle of quadratic differentials over moduli space. In future work, we hope to study the dynamics of this action on the bundle of \textit{meromorphic} quadratic differentials (see \cite{Boi} for comments on the non-ergodicity on lower-dimensional strata). \\

The strategy of the proof of Theorem \ref{thm1} is to consider a compact exhaustion of the punctured surface $X=\Sigma \setminus \mathcal{P}$, and construct a sequence of harmonic maps from their universal covers to the real tree that is the leaf space of the lift of the desired measured foliation. The main analytic work is to show that there is a convergent subsequence that yields a harmonic map whose Hopf differential is then the required meromorphic quadratic differential. 

 This strategy follows that of our previous paper \cite{GW1}, where the specific measured foliation we considered had a leaf-space that was a $k$-pronged tree, and we considered harmonic maps from the compact exhaustion to such a tree. One key difference is that now we pass to the universal cover, and  aim to obtain infinite-energy equivariant harmonic maps to a more general $\mathbb{R}$-tree. 
The basic analytic difficulties derive partly from the target being singular, but more significantly from the target having infinite co-diameter and the maps having infinite energy.

Also,  as mentioned before,  the singular-flat structure around the poles induced by an arbitrary meromorphic quadratic differential might comprise not only half-planes, but also horizontal strips. 
This features in the construction of the real trees -  in particular, the leaf-spaces of the foliations at the poles could now have finite-length edges dual to the strips, in addition to the infinite prongs that are dual to the half-planes; these lift to an equivariant collection of such edges in the real tree. 

The arrangement of strips destroy the apparent local symmetry of the maps near the poles. Nevertheless, a crucial new observation is that even in this general case, appropriate restrictions of the maps around the poles have enough ``symmetry" and can be thought of as branched covers of harmonic functions. We then adapt and streamline some key analytical results from our previous work (included in the present discussion for the convenience of the reader) to show the sub-convergence of the sequence of harmonic maps as desired.\\

\textbf{Acknowledgements.} Both authors gratefully appreciate support by NSF grants DMS-1107452, 1107263, 1107367 ``RNMS: GEometric structures And Representation varieties" (the GEAR network) as well as the hospitality of MSRI (Berkeley), where some of this work was initiated.  The second author acknowledges support of NSF DMS-1564374. The first author thanks the hospitality and support of the center of excellence grant `Center for Quantum Geometry of Moduli Spaces' from the Danish National Research Foundation (DNRF95) during May-June 2016.
The first author wishes to acknowledge that the research leading to these results was supported by a Marie Curie International Research Staff Exchange Scheme Fellowship within the 7th European Union Framework Programme (FP7/2007-2013) under grant agreement no. 612534, project MODULI - Indo European Collaboration on Moduli Spaces.

\section{Background}


\subsection{Quadratic differentials}

For this section, we shall fix  a compact Riemann surface $\Sigma$ of genus $g\geq 2$.

A \textit{holomorphic quadratic differential} on $\Sigma$ is a holomorphic section of the symmetric square $K^{\otimes 2}$ of the canonical line bundle on $\Sigma$, that is, a tensor locally of the form $q(z)dz^2$ for some holomorphic function $q(z)$, where $z$ is a complex coordinate on $\Sigma$. 

A meromorphic quadratic differential, correspondingly, is a meromorphic section of $K^{\otimes 2}$; the function $q(z)$ in the local expression may have poles of finite order.

Meromorphic quadratic differentials with poles at points $p_1,p_2,\ldots p_k \in \Sigma$ of orders bounded above by  $n_1,n_2,\ldots n_k \in \mathbb{Z}_+$ form vector space over $\mathbb{C}$, and it follows from the classical Riemann-Roch theorem that its dimension over $\mathbb{C}$  is $3g - 3 + \sum\limits_{i} n_i$.\\

Meromorphic quadratic differentials with poles of order at most one arise in classical Teichm\"{u}ller theory as the cotangent vectors to Teichm\"{u}ller spaces of punctured surfaces. Poles of order two also often arise as limits of ``pinching" deformations in which case one obtains noded Riemann surfaces (see, for example, \cite{Wolf3}).

 In this article we shall be interested in the case of \textit{higher order} poles, namely those with poles of order greater than two. (Poles of order two were treated in \cite{GW15}.)  Such differentials have also been considered more recently, as arising in more general limits of degenerations of Riemann surfaces, for example of singular-flat surfaces along a Teichm\"{u}ller ray (see, for example \cite{Gup3}). See also the recent work in \cite{Grushetal} 
 for the related case of compactification of strata of abelian differentials. 

\subsection{Measured foliations}\label{sec: measured foliations}

A holomorphic quadratic differential $q$ always admits \textit{canonical local charts} where the local expression of the differential away from the zeroes of $q$ may be written as $dw^2$; in terms of the original local expression $q(z)dz^2$, this change of coordinates is given by $ z\mapsto w =  \pm \displaystyle\int \sqrt q(z) dz$.

The new charts have transition functions $ z\mapsto \pm z+ c$ on their overlaps, and in particular induce:

\begin{enumerate}

\item[(A)] A \textit{singular-flat metric} pulled back from the Euclidean metric on the $w$-plane. Here the singularities are at the zeroes of $q$, where the above change-of-coordinate map is a branched covering.

In particular, at a zero of order $(n-2)$, the quadratic differential $z^{n-2}dz^2$ induces a branched cover $z\mapsto z^{n/2}=w$ on a canonical local chart so that the induced metric has cone-singularity of angle $n\pi$ and the horizontal foliation has an \textit{$n$-pronged singularity} (see Figure 1 - left, for the case of $n=3$). \\

\item[(B)]  A \textit{horizontal foliation} given by the horizontal lines in the local charts to the $w$-plane. 
Such a foliation is equipped with a transverse measure $\mu$: namely, an arc $\tau$ transverse to the horizontal measured foliation  is assigned a measure by the local expression
\begin{equation}\label{trans}
\mu(\tau) =  \displaystyle\int\limits_\tau  \left\vert    \Im \sqrt q dz\right\vert = \displaystyle\int\limits_\tau  \left\vert    \Im dw\right\vert
\end{equation}
which is well-defined since the transition maps are half-translations. 

\end{enumerate}

\textit{Remarks.}  (i) A measured foliation $F$ is a \textit{smooth} object; it can be defined on any smooth surface without reference to a complex structure or quadratic differential:
namely, $F$ is a smooth one-dimensional foliation with singularities of specific local forms (namely, induced by the kernel of $\Im (z^{n/2}dz)$ at the zeroes), and equipped with a measure on transverse arcs that is invariant under transverse homotopy. \\
(ii) A \textit{measured foliation} can also be also defined for meromorphic quadratic differentials (which are holomorphic away from the poles)  in a manner identical to (B) above; in addition to the singularities at the zeroes, there are ``pole-singularities"  we shall describe in the next section (\S2.3).\\

\textbf{Spaces of foliations.} For a closed surface, consider the space  of equivalence classes of  measured foliations as above (smooth with finitely many $n$-pronged singularities) on a closed surface,  where two foliations are equivalent if they determine the identical transverse measures on every simple closed curve ( up to transverse homotopy).

The space $\mathcal{MF}$ of  such measured foliations on a compact oriented surface of genus $g\geq 2$)  is in fact determined by the transverse measures of a  suitable \textit{finite} collection of arcs or curves, and $\mathcal{MF}\cong \mathbb{R}^{6g-6}$ (see \cite{FLP}).

Moreover, it is known (see Theorem 6.13 of  \cite{FLP}) that equivalent measured foliations differ by isotopy of leaves and Whitehead moves (which contract or expand leaves between singularities corresponding to zeroes). \\

In \S3 we shall parametrize the corresponding space of measured foliations with ``pole-singularities" (see the next subsection for more on the local structure of the foliation at the poles).\\

  \begin{figure}
  \centering
  \includegraphics[scale=0.45]{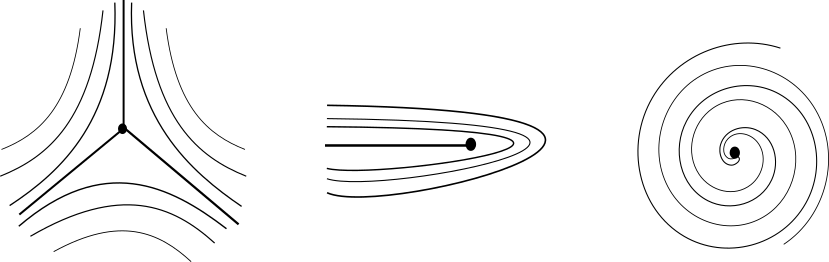}\\
  \caption{The foliations induced by $zdz^2$ (left), $\frac{1}{z}dz^2$ (middle),and $\frac{c}{z^2}dz^2$ (right) for a typical $c$. }
  \end{figure}
 
\textbf{Leaf spaces as $\R$-trees.}
Alternatively, an equivalence class of a measured foliation is specified by an $\mathbb{R}$-tree with a $\pi_1(S)$ action by isometries  that is \textit{small}, that is, stabilizer of any arc is a subgroup that does not contain any free group of rank greater than one. 

Here, an $\mathbb{R}$-tree is a geodesic metric space $(T,d)$ such that \textit{every} arc in the space is isometric to an interval in $\mathbb{R}$ - we refer to \cite{WolfT}, \cite{CullMor} for details and background.

Consider a measured foliation $\mathcal{F}$ induced by a holomorphic quadratic differential on a compact Riemann surface $X$ of genus $g\geq 2$. Lifting to the universal cover $\tilde{X}$, one obtains a one-dimensional smooth foliation $\tilde{\mathcal{F}}$ of the hyperbolic plane (with singularities) such that each leaf separates. In particular, each point of its \textit{leaf-space} $T$ is a cut-point, and when equipped with a metric $d$ induced from the transverse measures,  is in fact an $\mathbb{R}$-tree in the sense defined above. 

A \textit{collapsing map} shall be the map from $\tilde{X}$ to this metric space $(T,d)$ that maps each leaf of the induced foliation to the corresponding point in the leaf-space $T$. This map intertwines the action of $\pi_1(S)$ as deck transformations of $\tilde{X}$ with its action on the tree by isometries so that the resulting action on $(T,d)$ is a small action.

Note that by considering a measured foliation as a small $\pi_1(S)$ action on an $\R$-tree $(T,d)$, we automatically take care of the measure equivalence relation: neither an isotopy of the surface nor a Whitehead move affects the metric tree $(T,d)$ or the isometric action by $\pi_1(S)$.

 \begin{figure}
  \centering
  \includegraphics[scale=0.55]{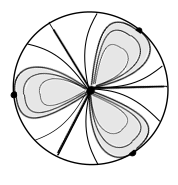}\\
  \caption{The foliation in a neighborhood of a pole of order $5$ as in Definition \ref{vp}. The distinguished points on the boundary $\partial U$  are where the half-planes (shown shaded) touch; in between them are the horizontal strips incident at the pole. }
  \end{figure}

\subsection{Structure at the poles of a quadratic differential}

We shall eventually wish to extend the discussion above from measured foliations with $m$-pronged singularities -- corresponding to neighborhoods of zeroes of holomorphic quadratic differentials -- to measured foliations with singularities that correspond to poles of meromorphic quadratic differentials.  Before jumping into that description, we first recall the pertinent theory of meromorphic quadratic differentials.

We first describe the correspondence between the local expression of a meromorphic quadratic differential at  a pole, and the induced horizontal measured foliation, that has a singularity there.

 At a pole of order $1$,  when the expression of the quadratic differential is  $\frac{1}{z}dz^2$, the change of coordinates $z \mapsto w = \sqrt z$ that converts the differential to $dw^2$ is a double cover branched at a point; the $1$-prong singularity is  thus a local ``fold" of the singular-flat structure in the regular case (see Figure 1 - center). 

For all poles of higher order, the area of the induced singular-flat metric is infinite, and the pole is at an infinite distance from any point on the surface.  For example, the singular flat metric induced near a pole of order two is isometric to a half-infinite Euclidean cylinder; the foliation comprises (typically spiralling) leaves along the cylinder towards the pole (see Figure 1 - right).

Finally, in a neighborhood of a  pole of \textit{higher order}, that is, of order greater than two,  that is the subject of this paper, we have the following two kinds of foliated sub-domains induced by the differential:

\begin{itemize}

\item (Half-planes)  Isometric to  $\{ z\in \mathbb{C}\vert  \mathrm{Im}(z) >0 \}$ 

\item (Horizontal strips) Isometric to  $\mathcal{S}(a) = \{ z\in \mathbb{C}\vert -a < \mathrm{Im}(z) <a\}$ for  $a\in \mathbb{R}_+$,
\end{itemize}
both in the standard Euclidean metric,  with the induced horizontal foliation being the horizontal lines $\{ \Im z = \text{constant}\}$.

\begin{defn}[Pole-singularities]\label{vp}  At a pole of order $n\geq 3$, it is known by the work of Strebel (\cite{Streb}) that there is a \textit{sink neighborhood} of the pole with the property that any horizontal leaf entering it continues to the pole in at least one direction.
Moreover,  any sufficiently small neighborhood of the pole contained in the sink neighborhood does not contain any $m$-pronged singularities (corresponding to the zeroes of the differential).
The local structure of the induced metric and horizontal foliation in such a neighborhood is an arrangement of half-planes and horizontal strips around the pole; in particular, there are exactly $n-2$ half-planes (that we also call \textit{sectors}) arranged in a cyclic order around the 
pole $p$. (See Figure 2).

For any choice of a smooth disk $U$ centered at a pole of order $n$ and contained in such a neighborhood, there are $n-2$ \textit{distinguished points} on $\partial U$, determined by points of tangency of the foliation in each sector, with the boundary of the disk. 
(The transverse measures of the arcs between the distinguished points are parameters of the measured foliation, as will be defined in Definition \ref{mf-pole}.) 

\end{defn}

\subsection*{Examples}

The quadratic differential $zdz^2$ on $\mathbb{C} \cup \{\infty\}$ which has a pole of order five at infinity, expressed as $\frac{1}{w^5} dw^2$ in coordinates obtained by the inversion $ z\mapsto w= 1/z$. 
The restriction of the foliation to the disk $U = \{ \lvert w \rvert <1\}$ is shown in Figure 2;  it can be calculated that the distinguished points on the boundary have equal angles between them, namely, they are at  $\{ -1, \pm e^{i\pi/3}\}$. This restriction of the foliation has three foliated half-planes around the singularity, with horizontal strips with one end incident at the singularity, between each.

An example where the presence of a horizontal strip is more apparent, is the quadratic differential $(z^2 -a) dz^2$ on $\mathbb{C}$, for some $a\in \mathbb{C}^\ast$, induces a singular-flat metric with four half-planes and a horizontal strip. (See Figure 3.)  In this case, the pole  of order six at infinity has the expression  $\left(\frac{1}{w^6} + \frac{a}{w^4}\right) dw^2$. (This pole has non-zero residue, which is introduced in the subsequent discussion.) \\

\begin{figure}
  \centering
  \includegraphics[scale=0.35]{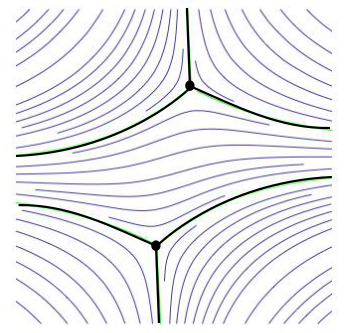}\\
  \caption{The measured foliation induced by a quadratic differential   $(z^2 -a) dz^2$ on $\mathbb{C}$ has a horizontal strip and four half-planes.}
  \end{figure}

\textbf{Measured foliations with pole singularities.} With the background on measured foliations with pronged singularities and meromorphic differentials behind us, we can now define our main object of interest:  measured foliations with pole singularities.

\begin{defn}[Measured foliations with pole singularities]\label{mf-pole} The space of measured foliations with pole-singularities, denoted  $\mathcal{MF}(n_1,n_2,\ldots n_k)$,  shall be the space of equivalence classes of measured foliations on a surface of genus $g\geq1 $ with $k$ points where it has singularities of the type of a pole of orders $n_1,n_2,\ldots n_k$, each greater than two. (See Definition \ref{vp}.) 

In this case, the equivalence relation is more refined, and is given relative to a choice of disk neighborhoods $U_1,U_2,\ldots U_k$ of the poles $p_1, p_2, \dots p_k$, respectively.  (We moreover assume that these disks are contained in the ``sink neighborhoods" of the poles - see Definition \ref{vp}.) Two such measured foliations with pole-singularities  on $S$ are equivalent if the transverse measures agree for the homotopy classes of
\begin{itemize}
\item all simple closed curves, 
\item loops around each pole,
\item simple arcs on the surface-with-boundary  $S\setminus U_1 \cup U_2 \cup \cdots \cup U_k$ with endpoints on the boundary components,  including, for each $i=1, \ldots k$, the $(n_i-2)$ arcs between the distinguished points on the boundary $\partial U_i$ 
(see Definition \ref{vp} for the definition of these points).
\end{itemize}
The $(n_i-2)$ parameters that are the transverse measures of the arcs between distinguished points on a suitable disk $U_i$ about the pole $p_i$ shall be referred to as the \textit{local parameters} at the pole. 

\end{defn}

\textit{Remark.} As mentioned at the end of \S2.2, an equivalence class of a measured foliation on a closed surface is captured by the (metrized) leaf-space in the universal cover, that is an $\mathbb{R}$-tree with a small action of $\pi_1(S)$.  Similarly, for a measured foliation with pole-singularities, the (metrized) leaf-space in the universal cover determines an $\mathbb{R}$-tree that uniquely determines the equivalence class of the measured foliation. This time, the leaf-space of the foliated half-planes around the poles determine an equivariant collection of infinite rays (or ``prongs"), and the horizontal strips, when present, add finite-length edges between them. The structure of the leaf-space in a neighborhood of the poles is described in detail in \S3.\\

\textbf{Residue.} We now turn to some of the \textit{analytical} data one can define at the pole:

 The \textit{residue} of the quadratic differential $q$ at a pole is the integral of $\pm\sqrt q$ along a simple loop around the pole.  Note that this is defined up to an ambiguity of sign, which we prefer to keep. 
This residue is a coordinate-independent complex number, and vanishes for odd order poles.

In fact, by a result of Strebel (see \S6 of \cite{Streb}) the meromorphic quadratic differential with a pole of order $n$ at the origin has the following  ``normal form" with respect to some coordinate $z$:
\begin{equation}\label{nform1}
\frac{1}{z^{n}}dz^2 \text{  when } n \text{ is odd}
\end{equation}
and 
\begin{equation}\label{nform2}
\left(\frac{1}{z^{n/2}} + \frac{a}{z}\right)^2dz^2 \text{  when } n \text{ is even.}
\end{equation}
where $\pm a \in \mathbb{C}$ is the residue that we just defined.\\

We shall use these normal forms to define a ``symmetric exhaustion" near the pole in \S4.2.\\

\textbf{Principal part.} Given a meromorphic quadratic differential $q$ with a pole of order $n$ at $p$,  for an arbitrary choice (i.e. a choice not {\it a priori} adapted to the differential) of coordinate $z$ around $p$, we have the following notion of a \textit{principal part} for $\sqrt q$ (up to a choice of sign) at the pole:\\
The differential $\sqrt q$ has the expression:
\begin{equation}\label{princ1}
\frac{1}{z^{n/2}}\left(P(z) + z^{\frac{n}{2}}g(z) \right) dz
\end{equation}
when $n$ is even, where $g(z)$ is a non-vanishing holomorphic function, and $P(z)$ is a polynomial of degree $\frac{n-2}{2}$ , and
\begin{equation}\label{princ2}
\frac{1}{z^{n/2}}\left(P(z) + z^{\frac{n-1}{2}}g(z) \right) dz
\end{equation}
when $n$ is odd, where $g(z)$ is a holomorphic function as before, and $P(z)$ is a polynomial of degree $\frac{n-3}{2}$.

In either case, the polynomial $P(z)$ shall be the \textit{principal part} of $\sqrt q$; note that this is determined by the $\frac{n}{2}$ (resp. $\frac{n-1}{2}$) complex parameters (namely the coefficients), for $n$ even (resp. odd), where constant term is also non-zero.  Note also that the ambiguity of sign means that principal part formed by taking the negative of all parameters is considered the same.  \\

 \textbf{Compatibility.} We now note that the real part of the residue at a pole of even order is determined by the parameters of the induced horizontal foliation at the pole. (At an odd order pole, the residue is zero, as can be seen by using \eqref{nform1}.) 
  
 It can be checked from the expression \eqref{nform2} of the normal form at an even-order pole that the disk $U = \{ \lvert z\rvert < \lvert a \rvert ^{-1}\}$ does not contain any zeroes of the differential, and is thus a sink neighborhood as in  Definition \ref{vp}. In particular, the $(n-2)$  points at equal angles on $\partial U$, including one at $e^{i\pi/(n-2)}$, are the ``distinguished points" where the leaves of the horizontal foliation are tangent to the boundary $\partial U$.
 
 Recall from Definition \ref{mf-pole} that the transverse measures of the arcs $\gamma_1, \gamma_2,\ldots \gamma_{n-2}$ between these distinguished points are the ``local parameters" of the induced horizontal measured foliations.

 \begin{lem} In the setting above, the real part of the residue at the pole of even order $n\geq 4$ is the alternating sum of the transverse measures of the distinguished arcs, that is, we have 
 \begin{equation}\label{comp-eq}
 \sum\limits_{j=1}^{n-2} (-1)^j \mu(\gamma_j) = 2\pi \Re(a).
  \end{equation}
  (Note that the left hand side also has an ambiguity of sign, since it depends on the distinguished arc that one starts with.) 
 \end{lem}
  
 \begin{proof}
 By \eqref{nform2}, the quadratic differential is locally the square of a holomorphic one-form $\pm \omega$ at the pole.  We fix a sign for $\omega$.
  
 Let $\gamma$ be a simple closed loop linking the pole. Note that $\gamma$ is a concatenation of the distinguished arcs $\gamma_1,\ldots \gamma_{n-2}$, which are circular arcs of angle $2\pi/(n-2)$ as described above.
 
An easy computation shows that the signs of the integrals $ \displaystyle\int_{\gamma_j} \Im \sqrt{q}$, and indeed, the integrands $\Im\omega(\gamma_j(t)) \gamma_j^\prime(t)$ (for, say, an arclength parametrization of $\gamma_j)$, alternates for $j=1,2, \ldots , n-2$. By (\ref{trans}), the transverse measure of the arc $\gamma_j$ is then the absolute value of such an integral.

Hence we have:
\begin{equation}\label{mu}
\displaystyle\int_{\gamma} \Im \sqrt{q}  = \sum\limits_{j=1}^{n-2}  (-1)^j \displaystyle\int_{\gamma_j} \Im \sqrt{q}   =   \sum\limits_{j=1}^{n-2} (-1)^j \mu(\gamma_j)  
\end{equation}

However from the expression \eqref{princ1} we can calculate:

 \begin{equation}\label{mu1}
\displaystyle\int_{\gamma} \Im \sqrt{q} = \Im  \displaystyle\int_{\gamma} \frac{a}{z} dz   =  2\pi \Re (\pm a) 
\end{equation}
and we obtain the desired identity. 
\end{proof}

As mentioned in the Introduction,  we can now define:

\begin{defn}\label{compat}
A meromorphic quadratic differential on surface of  genus $g\geq 2$ has principal parts  \textit{compatible} with a measured foliation (with singularities corresponding to poles) if at each pole of even order, the equality in (\ref{comp-eq}) holds, that is, the real part of the residue agrees with that determined by the ``local parameters"  of the horizontal foliation around the pole (\textit{cf.} Definition \ref{mf-pole}). 

Note that at any pole of odd order, the residue vanishes, and there is no additional requirement for compatibility. 
\end{defn}

 Let $\mathsf{Comp}(F)$ be the space of principal parts compatible with a fixed foliation $F\in \mathcal{MF}(n_1,n_2,\ldots n_k)$. Then we have:
 
 \begin{lem}\label{comp} The space of compatible principal parts $\mathsf{Comp}(F)$ is  homeomorphic to $\prod_{i=1}^k(\mathbb{R}^{n_i-2}\times S^1)$.
 \end{lem}
 \begin{proof}  It suffices to prove the case of a single pole of order $n$, namely when $k=1$, as the general case is obtained as a $k$-fold cartesian product of the spaces defined in this simple case. 
 For $n$ even, the principal part  $P(z)$ as in (\ref{princ1}) is determined by  $n/2$  complex coefficients, and each contributes two real parameters except the (non-zero) top coefficient  which contributes an $\mathbb{R} \times S^1$, and the coefficient of $z^{-1}$, that contributes only an $\mathbb{R}$ in order  to match the given real residue. For $n$ odd, the residue is automatically zero; however there are $(n-1)/2$ terms in the principal part, and we obtain the same parameter space as above.
 \end{proof}

 Our main theorem (Theorem \ref{thm1}) asserts that these compatible principal parts can be \textit{arbitrarily} prescribed for a meromorphic quadratic differential realizing a given measured foliation.

\subsection{Harmonic maps to $\R$-trees}
We recall the initial discussion of $\mathbb{R}$-trees from \S\ref{sec: measured foliations}, together with the collapsing map from the universal cover $\tilde{X}$ along the leaves of a measured foliation to a metric tree $(T,d)$.

A useful observation (see \cite{WolfT}, see also \cite{DaWen}) is that the collapsing map  is \textit{harmonic} in a sense clarified below; note that locally, away from the singularities, the map is $z\to \Im (z)$, which is a harmonic function.\\

A \textit{harmonic map} $h$ from  Riemann surface $X$  to an  $\mathbb{R}$-tree  is  a critical point of the energy-functional 
\begin{equation*}
\mathcal{E}(f) = \displaystyle\int\limits_{X} \lVert df\rVert^2 dzd\bar{z}
\end{equation*}
on the space of all Lipschitz-continuous maps  to the tree. (Note that the energy density in the integrand is defined almost-everywhere for such maps.)  For \textit{equivariant} harmonic maps from the universal cover $\tilde{X}$, we consider the \textit{equivariant energy}, namely the above integral over a fundamental domain. Moreover, when the fundamental domain is non-compact and the energy is infinite, as is the case in this paper, we emphasize that one restricts to \textit{compactly supported} variations in characterizing a map that is critical for energy. 

 For more on theory of harmonic maps to NPC metric spaces, we refer to  Korevaar-Schoen (\cite{KorSch}). For $\mathbb{R}$-tree  targets we could alternatively use an equivalent characterization (see \S3.4 of \cite{FarWol} and Theorem 3.8 of \cite{DaWen}), namely that an equivariant map from the universal cover $\tilde{X}$ to  $T$ is \textit{harmonic} if locally, germs of convex functions pullback to germs of subharmonic functions.

\subsection*{Hopf differential}
We have described above how to obtain a measured foliation from a holomorphic quadratic differential, how to obtain a tree from a measured foliation, and how to obtain a harmonic map from the data of a map from a Riemann surface to a tree: we next complete this circle of relationships by describing how to obtain a holomorphic quadratic differential from a harmonic map.

In passing from holomorphic quadratic differential to measured foliation to tree to harmonic map, the original holomorphic quadratic differential can be recovered (up to a fixed real scalar multiple) by taking the \textit{Hopf differential} of the (collapsing) harmonic map $h$, which is locally defined to be 
\begin{equation*}
\Hopf(h) =  -4\left( \frac{\partial h}{\partial z}\right)^2 dz^2
\end{equation*}

\textit{Remark.} The constant in the definition above is chosen such that the Hopf-differential of the map $z\mapsto \Im z$ is $dz^2$, This sign convention  differs from the one used in, say \cite{Wolf2}, and in other places in the harmonic maps literature. In this convention, the  geometric interpretation of the  Hopf differential is that the \textit{horizontal} foliation are integral curves of the directions of minimal stretch of the differential $dh$ of the harmonic map $h$. (See e.g. \cite{Wolf0} for the computations justifying this.) \\

 Thus, when the harmonic map $h$ is  a projection to a tree,  the minimal stretch direction lies along the kernel of the differential map $dh$, and thus the horizontal leaves are the level sets of points in the tree. Away from the isolated zeroes of the Hopf differential, the harmonic map takes disks to geodesic segments in the $\mathbb{R}$-tree $T$.

Using this analytical technique, the second author re-proved the Hubbard-Masur Theorem in \cite{Wolf2}. 
As discussed in the Introduction, the strategy of this paper is to construct meromorphic quadratic differentials by considering Hopf differentials of  infinite-energy harmonic maps to $\mathbb{R}$-trees  that are dual to measured foliations with pole-singularities.  \\

We conclude this section with two well-known facts about harmonic maps that shall be useful later. References for these lemmas include \cite{Jost1} and \cite{KorSch1} (\textit{eg.} pg. 633) and \cite{KorSch2}; see also the discussion in \cite{Wolf2}.

\begin{lem}\label{dist}
Let  $T$ be a metric $\mathbb{R}$-tree.  Then  an equivariant harmonic map $h:\tilde{X}  \to T$ post-composed with the distance function from a fixed point $q\in T$  is subharmonic. 
\end{lem}

\begin{lem}\label{cour} A sequence of harmonic maps $h_i:\tilde{X} \to T$ equivariant with respect to a fixed isometric action of $\pi_1X$ on $T$,  having a uniform bound on energy on compact subsets, forms an equicontinuous family.  
\end{lem}

\begin{proof}[Sketch of the proof]

This follows from the Courant-Lebesgue Lemma, that provides a modulus of continuity for a harmonic map of bounded energy, when the domain is two-dimensional. We recount the argument briefly:  From the bound on total energy, the energy of the maps  restricted to a small annulus is uniformly bounded. Considering the diameter of images of circles in that annulus, a rewriting of the total energy in the annulus as an integral over concentric circles then forces one of the circles $C_0$ in the annulus to have small diameter. Because the tree is an NPC space, the domain interior to the chosen circle have maps within the convex hull of $h(C_0)$, in this case a subtree of $T$. This proves equicontinuity for the sequence of maps on this subdomain. \end{proof}

\textit{Remark.}
In several of our constructions, we will need to find a harmonic map to an $\R$-tree as a limit of a sequence of maps whose energies are tending to an infimum. Lemma~\ref{cour} above will provide equicontinuity for that family.  The Ascoli-Arzela theorem then provides for the subconvergence of the family of maps, say $u_i$, to a harmonic map $u$  if we also know that the images $u_i(p)$ of each point $p$ in the domain lie in a compact set.
This is a particular challenge when the target is an $\R$-tree, as such trees are typically not locally compact, so even a bound on the diameter of $\{u_i(p)\}$ does not suffice to guarantee the compactness of the convex hull of that set.
In general, powerful results of Korevaar-Schoen \cite{KorSch1} and \cite{KorSch2} are useful for proving existence of energy-minimizing maps in non-locally compact settings, but in the context of infinite-energy harmonic maps as in this paper, the exact statements we needed were not available in those references, and we develop some {\it ad hoc} methods (that work in the narrow context of  $\R$-tree targets) in \S4 (see, for example,  the proof of Proposition~\ref{prop: h_i convergence}).

\section{Space of measured foliations with poles}

In this section we shall parametrize $\mathcal{MF}(n)$, the space of measured foliations on $S$ (up to measure-equivalence) with one pole-singularity defining $(n-2)$-sectors at a point $p$, where $n\geq 3$. (For the terminology and definitions, see \S2.2.)  This is done in Proposition \ref{comb}.  The arguments in this section extend to the case of more than one pole of higher order, which we omit for the sake of clarity.  The main result that the extension yields is:

\begin{prop}\label{prop:mfk} Let $S$ be a closed surface of genus $g\geq 2$, and let $n_i\geq 3$ for $i=1,2,\ldots k$. Then the space of foliations 
$\mathcal{MF}(n_1,n_2,\ldots n_k)$ is homeomorphic to  $\mathbb{R}^{\chi}$, where $\chi = {6g-6 + \sum\limits_i (n_i +1)}$.
\end{prop}

The idea of the proof is to ascertain parameters for the restrictions of a foliation in $\mathcal{MF}(n)$  to a neighborhood $U \cong \mathbb{D}$  of the pole, and to its complement $S \setminus U$. The number of real parameters for the former is $n-1$ (see Proposition \ref{prop:pk}) and the latter is $6g-6 +3$ (Proposition \ref{mfgb}), and there is one parameter (the transverse measure of $\partial U$) that should agree for the two foliations to yield a measured foliation on the surface $S$. 

To determine these parameters we consider the $\mathbb{R}$-trees that are the leaf-spaces of the measured  foliations when lifted to the universal cover. We recall that one advantage of considering these dual trees is that the trees ignore the ambiguities associated to Whitehead moves and so reflect only the equivalence classes of the foliations. Another advantage, and one we will exploit in this section, is that one can use standard results parametrizing the \textit{spaces} of such metric trees (see Theorem \ref{mulpenk} of  \S3.2).  To build such an $\mathbb{R}$-tree from the data of the parameters, the leaf space $T$ for the lift of the restriction to $S\setminus U$ and  the leaf-space $T_U$ for the lift of the foliation on $U$ are equivariantly identified along the the lifts of $\partial U$.

In the next section (\S3.1), we recall a parametrization of the measured foliations on a surface-with-boundary culled from \cite{ALPS}; in our case we shall apply it to $S\setminus U$. In \S3.2 we describe the local models and dual trees for the foliation on a disk-neighborhood  $U$ of the puncture, and provide a parametrization. Finally, in \S3.3, we describe the ``identification"  of the trees mentioned above and complete the proof of Proposition \ref{comb}, the single pole version of Proposition~\ref{prop:mfk}. As noted, that proof generalizes to the case of multiple poles stated in Proposition~\ref{prop:mfk}.

\subsection{Surface with boundary}

Recall that for a closed  surface, the space of measured foliations  $\mathcal{MF}$ can be given coordinates by Dehn-Thurston parameters, that are the transverse measures (``lengths" and ``twists") about a system of pants curves. (For details see \cite{DehnThur}.) 
 When the surface is compact with non-empty boundary, one can reduce it to the closed case by doubling across the boundaries, and taking a system of pants curves that includes the closed curves obtained from the doubled boundaries.  This yields the following parametrization - see \cite{ALPS} for a description, and Proposition 3.9 of that paper for the proof. We provide a sketch of the argument for the reader's convenience. 

\begin{prop}\label{mfgb} The space of measured foliations (without poles) $\mathcal{MF}_{g,b}$ on a surface of genus $g$ with $b$ boundary components is homeomorphic to $\mathbb{R}^{6g-6 + 3b}$.
\end{prop}
\begin{proof}[Sketch of proof]
For simplicity, consider the case of a  foliation $F$ with transverse measure $\mu$, on a surface with a single boundary component. By an isotopy of the foliation relative to the boundary, one can assume that the leaves of $F$ are either all parallel,  or all orthogonal, to the boundary. By doubling across the boundary, one obtains a measured foliation $\hat{F}$ on a compact surface with an involutive symmetry. Such measured foliations can be parametrized by Dehn-Thurston parameters (transverse measures and twists) with respect to a  symmetric pants decomposition that includes the curve representing the (doubled) boundary.  Other than this doubled boundary, there are then a total of $6g- 4$  ``interior" pants curves. Each ``interior" pants curve $C$  contributes two parameters: the transverse measure $\mu(C) \in \mathbb{R}_{\geq 0}$ and the ``twist" coordinate $\theta(C) \in \mathbb{R}$ (see Dylan Thurston's article  \cite{DehnThur} for details). This gives a parameter space $(\mu(C), \theta(C)) \in \mathbb{R}_{\geq 0} \times \mathbb{R}$, 
and together with the identification $(0,t) \sim (0, -t)$, this gives a parameter space  $\mathbb{R}_{\geq 0} \times \mathbb{R}/\sim$ that is homeomorphic to $\mathbb{R}^2$.

 Taking into account the involutive symmetry, exactly $3g-2$ interior pants curves contribute independent pairs of parameters. 
 
 For the curve $B$ corresponding to the doubled boundary, the ``twist" coordinate vanishes. Either the curve $B$ has positive transverse measure $i(B) = \mu(B)$, or  $\hat{F}$ contains a foliated cylinder of closed leaves parallel to $B$, and we can define $i(B) = - L$ where $L$ is the transverse measure across the cylinder. This yields a parameter $i(B) \in \mathbb{R}$.
 
 With two real parameters from each of the $3g-2$ interior pants curves, and one from the boundary curve, we obtain the parameter space described in the statement. 
\end{proof}

\subsection{Foliations around a pole}

Consider a coordinate disk $U$ centered at the pole singularity  $p$ of order $n\geq 3$.  First, we define a space of ``model" measured foliations on $U$:

\begin{defn}[Model foliations]\label{defn:pn} The space $\mathcal{P}_{n}$ of measured foliations  on  $U \cong \mathbb{D}$ comprises equivalence classes of measured foliations with a pole-singularity at the origin with $(n-2)$ sectors as in Definition \ref{vp},  such that any leaf has at least one endpoint at the origin. In the terminology of Definition \ref{vp}, the disk $U$ is contained in the sink neighborhood of the pole. Here, two measured foliations in $\mathcal{P}_{n}$ are equivalent if they differ by Whitehead moves or an isotopy of leaves relative to the boundary $\partial {U}$. 
\end{defn}

\textit{Remark.} In the definition above we do not assume that $U$ is free of any $m$-pronged singularities; for a measured  foliation in $\mathcal{P}_{n}$, its  foliated half-planes and horizontal strips may be arranged with such singularities lying on their boundaries. In particular, the horizontal strips could be bi-infinite in the induced singular flat metric on $U \setminus p$ (running from the pole to itself, and contained in $U$), or half-infinite (one end goes to the puncture, the other gets truncated at the boundary $\partial U$.   This gives rise to different combinatorial possibilities, which are involved in the parametrization that follows. \\

Heuristically, the parameters for these measured foliations on $U$ are the transverse measures of the $(n-2)$ strips, together with the total transverse measure of the boundary.  In what follows, we shall work with the tree that is the leaf-space of the lift of the foliation to the universal cover; the transverse measures then provide the metric on the tree. The different combinatorial possibilities of arrangements of the half-planes and strips are then reflected in the combinatorial structure of the tree.  (See Figure 4.) 

This tree shall form part of the real tree that determines the equivalence class of a measured foliation with a pole-singularity on the surface $S$ (see the remark following Definition \ref{mf-pole}). \\

 \begin{figure}
  \centering
  \includegraphics[scale=0.5]{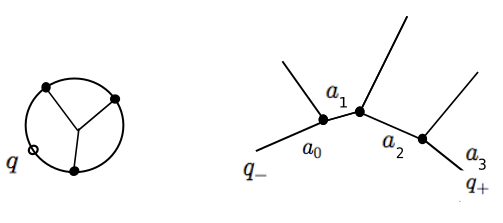}\\
  \caption{The figure on the left is a  metric graph that is a leaf-space of a model foliation on $U$ (see Definition \ref{defn:pn}). The leaf-space of the foliation restricted to the punctured disk $U \setminus p$ lifts to a tree with a $\mathbb{Z}$-action; the figure on the right is one possible  fundamental domain of the action.}
  \end{figure}

We first recall a definition and a result from \cite{MulPenk}:

\begin{defn}\label{metx}  A \textit{metric expansion} of a tree $T$ with a single vertex $v$ with $\mathcal{V} \geq 4$ labelled edges emanating from it,  is obtained by replacing the vertex with a metric tree.  (See Figure 5.) 
\end{defn}

\begin{thm}[Theorem 3.3 of \cite{MulPenk}]\label{mulpenk} The space of metric expansions of a metric tree $G$ at a vertex $v$ of valence $\mathcal{V} \geq 4$  is homeomorphic to $\mathbb{R}^{\mathcal{V}-3}$. 
\end{thm}

\textit{Remark.} A metric tree obtained by a generic metric expansion at a vertex of valence $\mathcal{V}$  has $\mathcal{V}-3$ edges of  finite length. 
For each fixed isomorphism type of a generic graph, these edge-lengths then form a  parameter space homeomorphic to $\mathbb{R}_+^{n-1}$. Different isomorphism types of such metric trees are obtained by Whitehead moves on the finite-length edges;  the corresponding parameter spaces fit together to form a cell.

\begin{prop}\label{prop:pk} The space of model foliations $\mathcal{P}_n$ is homeomorphic to $\mathbb{R}^{n-3} \times \mathbb{R}_{\geq 0}$. 
\end{prop}

\begin{proof}

For a foliation in $\mathcal{P}_n$ consider its restriction to the punctured disk $U \setminus p$, where $p$ is the pole singularity.

In the universal cover of the punctured disk $U\setminus p$, the lift of this foliation thus has a leaf-space that is a metric tree $Y$ with a $\mathbb{Z}$-action by isometries (here $ \mathbb{Z} \cong \pi_1(U \setminus p))$.  Each edge dual to the lift of a horizontal strip acquires a finite length that is the transverse measure, or Euclidean width, of the strip.
Note that the $\mathbb{Z}$-action preserves each connected component of the lift of $\partial U $, and if such a component is a bi-infinite path (and embedded $\mathbb{R}$), the restriction of the action to the component is by translations. 

It suffices to parametrize the metric trees that are the possible fundamental domains for such an action.\\
 
  \begin{figure}
  \centering
  \includegraphics[scale=0.45]{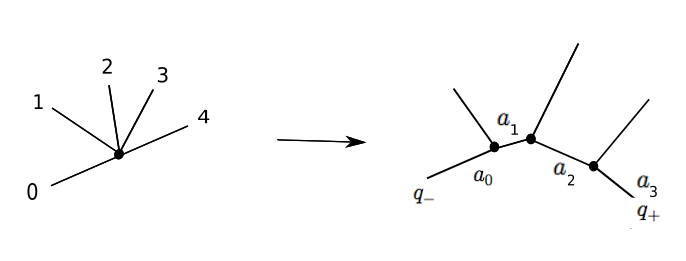}\\
  \caption{The fundamental domain of the $\mathbb{Z}$-invariant leaf-space of the foliation in $\mathcal{P}_n$ when lifted to $\widetilde{U \setminus p}$ can be obtained for the metric expansion of a valence $n$ vertex when $\partial U$ has positive tranverse measure. The labelling in  cyclic order comes from the surface orientation. }
  \end{figure}

 \textit{Case I.  The boundary $\partial U $ has positive transverse measure.}

Note that the $\mathbb{Z}$-action on the tree is by translation along an infinite axis that is a concatenation of finite-length edges. Any  fundamental domain $D$  has $(n-2)$ infinite rays with a labelling based on a cyclic ordering acquired from the orientation of the surface, and in the generic case, $D$ has $(n-1)$ finite-length edges; this is one more than the number of edges in the leaf-space on $U \setminus p$ because of a choice of basepoint $q$ on $\partial U$ that bounds the fundamental domain and, generically, separates an edge between two infinite rays (see Figure 5 for the case when $n=5$).  We choose the basepoint $q$ on an edge dual to a non-degenerate strip; the assumption of positive transverse measure implies that not all strips are degenerate.

The two extreme points of this axial segment in $D$ are then $q_-$ and $q_+$ respectively, where the latter point is the image under translation of the former one, and we include only one of them in $D$.

The length of the path between $q_-$ and $q_+$ is then the  translation distance that is the transverse measure of the boundary.

The space of such metric trees can be obtained from metric expansions of a single tree:

Namely, let $T_U$ be a metric tree with a single vertex of valence $n$, and edges are labelled  $\{0,1,\ldots n-1\}$ in a cyclic order. (See Figure 5.)

Any metric tree that can be a fundamental domain $D$  then arises by considering an expansion of $graph$,  and assigning an additional edge lengths $a_0$ and $a_{n-1}$ to the edges labelled $0$ and $n-1$, that are adjacent to the endpoints $q_-$ and $q_+$ respectively, such that $a_0+ a_{n-1} >0$.  Here, the non-negative edge-lengths $a_0$ and $a_{n-1}$ cannot \textit{both} be zero as we had chosen the basepoint $q$ to lie in the interior of an edge dual to a non-degenerate strip. 

Since $q_-$ and $q_+$ are identified by the translation, it is this latter sum that is the real parameter of interest. 

The metric expansions contribute $(n-3)$ real parameters, by Theorem \ref{mulpenk}, and the two additional edge lengths $a_0$ and $a_{n-1}$ contribute one.  Together, the space of possible measured foliations with positive transverse measure is then homeomorphic to $\mathbb{R}^{n-3} \times \mathbb{R}_+$ where we consider the last parameter to be the total transverse measure (or translation length). \\

 \begin{figure}
  \centering
  \includegraphics[scale=0.45]{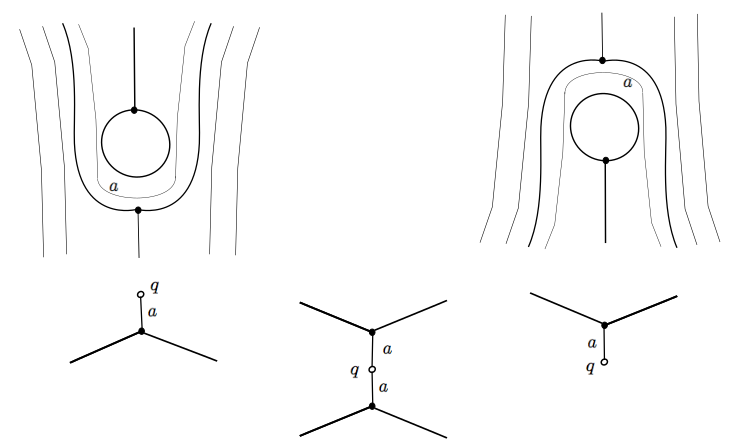}\\
  \caption{The figures on top show measured foliations in $\mathcal{P}_4$ (with the pole at $\infty$)  with the boundary of $U$ having zero transverse measure. Their leaf-spaces (bottom) have a single edge and two infinite rays; here $a$ is the transverse measure across the strip. The tree in the center is obtained by doubling this leaf-space, and is a metric expansion of a valence-$4$ vertex. }
  \end{figure}

 \textit{Case II.  The boundary $\partial U $ has zero transverse measure.}

In this case the $\mathbb{Z}$-action on the leaf-space has no translation distance, and fixes a distinguished point $q$.  The fundamental domain $D$ is then a planar metric tree, with $(n-2)$ infinite prongs that acquire a cyclic labelling from the orientation of the surface, together with the distinguished vertex $q$ that is a ``root" of the tree.

Doubling the fundamental domain $D$ across $q$, we obtain a metric expansion of a tree with a valence $(2n-4)$ vertex having an involutive symmetry.  

By Theorem \ref{mulpenk}, if the edge-lengths are arbitrary, then the parameter space is $\mathbb{R}^{2n-7}$. However, except the edge containing $q$, the other edge-lengths occur in pairs because of the involutive symmetry. Hence the number of  real parameters arising from these metric expansions is $\frac{2n-7-1}{2} +1 = n-3$. 
Here, the edge incident on $q$ is dual to a bi-infinite strip adjacent to the boundary, and its edge-length is the transverse width of the strip. This real parameter takes both positive and negative values, where the negative values correspond to the bi-infinite strip being asymptotic to a different labelled direction. (See Figure 6.) \\

Together, Cases I and II fit together to give a total parameter space that is then $\mathbb{R}^{n-3} \times \mathbb{R}_{\geq 0}$: note that each case results in a parameter space whose first factor is $\mathbb{R}^{n-3}$, while the second factors are $\mathbb{R}_+$ and $\{0\}$ whose union is the desired $\mathbb{R}_{\geq 0}$.
\end{proof}

For the identification of the foliated disk $U$ with $S\setminus U$ along the boundary $\partial U$, there is an additional circle's parameter of gluing in the case when the boundary has positive transverse measure. It is convenient to record this as part of the data of the foliated measured foliation on $U$, and we denote by $\widehat{\mathcal{P}}_n$  the resulting space of \textit{pointed} measured foliations. A consequence of the preceding parametrization is:

\begin{cor}\label{hatpn}
The space of pointed measured foliations $\widehat{\mathcal{P}}_n$  is homeomorphic to $\mathbb{R}^{n-1}$.
\end{cor}
\begin{proof}
The non-negative real factor $\mathbb{R}_{\geq 0}$  in the parametrization for $\mathcal{P}_n$ can be thought of as the transverse measure. For each positive measure $\tau$  the circle's ambiguity provides a parameter space $(\tau, \theta)  \in \mathbb{R}_+ \times S^1$, which can be thought of as polar coordinates for the punctured plane. The case of zero transverse measure contributes a point that fills the puncture. Thus,   $\widehat{\mathcal{P}}_n$ is homeomorphic to $\mathbb{R}^{n-3} \times \mathbb{R}^2$ = $\mathbb{R}^{n-1}$.
\end{proof}

\subsection{Parametrizing $\mathcal{MF}(n)$}

We now combine the results of the preceding sections and prove the following paramaterization.  In course of the proof, we give a careful description of the  $\mathbb{R}$-tree, dual to the lift of a measured foliation with pole-singularities to the universal cover, that shall be the non-positively curved (NPC) target of the harmonic map we construct in \S4.

\begin{prop}\label{comb}
The space $\mathcal{MF}(n)$ is homeomorphic to $\mathbb{R}^{6g-6 + n+1}$.
\end{prop}
\begin{proof}
Fix a disk $U$ about the singularity. We shall now combine a measured foliation on $S\setminus U$ and a foliation $F_U \in \widehat{\mathcal{P}}_n$ (\textit{cf.} Corollary \ref{hatpn}) to produce a measured foliation in $\mathcal{MF}(n)$. 

The process of combining the foliation $F\in \mathcal{MF}_{g,b}$ on $S\setminus U$ and the model foliation $F_U$ on $U$ can be described as an equivariant identification of their dual $\mathbb{R}$-trees (in their lifts to the universal cover). We now describe this in detail. 

For the surface with boundary $S\setminus U$, let $T_0$ be the $\mathbb{R}$-tree that is the leaf-space of the lifted foliation $\widetilde{F}$ on $\widetilde{S\setminus U}$, and let
\begin{equation*}
\pi: \widetilde{S\setminus U} \to T_0
\end{equation*}
be the projection map that takes each leaf of the lifted foliation $\widetilde{F}$ to a unique point in $T_0$, which is equivariant with respect to an action of $\pi_1(S\setminus U)$.

In particular, the circle $\partial U$ lifts to an equivariant collection of real lines $L_\alpha$ (for $\alpha$ in some index set $I$) in the universal cover $\widetilde{S\setminus U}$. Each $L_\alpha$ maps into the tree $T_0$ injectively as an embedded  axis that we denote by $l_\alpha$ in the case when $\partial U$ has positive transverse measure, and to a point that we denote by $q_\alpha$ when $\partial U$ has zero transverse measure.

 By the equivariance of the above projection map $\pi$, in the case when the transverse measure is positive, the infinite cyclic group generated by the boundary circle $\langle \partial U \rangle < \pi_1(S \setminus U)$  acts on each such line $l_\alpha$ by a translation with non-trivial translation distance equal to the transverse measure. In the case when the transverse measure of the boundary is zero, the action of this $\mathbb{Z}$-subgroup fixes the vertex $q_\alpha$.

We may restrict the pointed foliation $F_U \in \widehat{\mathcal{P}}_n$ on $U$ to the punctured disk $U \setminus p$. On the universal cover $\widetilde{U \setminus p}$ of $U \setminus p$, the lifted foliation has a leaf space that is a metric tree $T_U$, and there is a projection $\pi_U$ equivariant under the action of $\pi_1(U \setminus p) =  \mathbb{Z}$. This time the boundary $\partial U$ lifts to a single line $L$ in $\widetilde{U \setminus p}$. In the case the boundary has positive transverse measure, this embeds in the tree $T_U$ as a bi-infinite axis $l_U$, and the equivariance provides a $\mathbb{Z}$-action corresponding to the infinite cyclic subgroup generated by $\langle \partial U\rangle = \pi_1(U\setminus p)$ that acts by translation along $l_U$. In the case the boundary has zero transverse measure, the image of $L$ is a single vertex $p\in T_U$, which is fixed by the $\mathbb{Z}$-action.  (See Figure 7.) 

\begin{figure}
  \centering
  \includegraphics[scale=0.63]{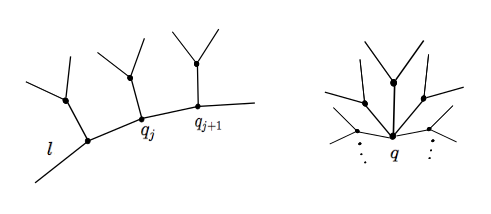}\\
  \caption{The leaf-space $T_U$ in the case when $\partial U$ has positive transverse measure (figure on left), and in the case when $\partial U$ has zero transverse measure (figure on right).  In the former case the $\mathbb{Z}$-action is by translation along a bi-infinite axis $l$, and in the latter case, the $\mathbb{Z}$-action fixes the point $q$.}
  
    \end{figure}

We now describe the procedure of identifying the metric trees $T_0$ and $T_U$ along the axes (or points) corresponding to the image of the lifts of $\partial U$ under the projections to the leaf-space.

This identification can only be done when the transverse measures of $\partial U$ for the measured foliations on $S\setminus U$ and $U$ are identical.

In the case of positive translation measure, each bi-infinite axis $l_\alpha$ in $T_0$ is identified with the axis $l$ in a copy of $T_U$, such that the $\mathbb{Z}$-actions, that acts by translations by the same distance on the axes, agree. 

Typically there is a circle's worth of ambiguity in this identification; this is where the fact that the foliation $F_U$ is \textit{pointed} is helpful; in particular, it provides a distinguished point $q$ on the circle $\partial U$.

Consider the orbit $\{ q_{j, \alpha} \vert j \in \mathbb{Z}\}$ of the action on each $l_\alpha \subset T_0$ that are lifts of this point $q\in \partial U  = \partial (S\setminus U)$. Similarly, there is an orbit $\{ q_{j} \vert j \in \mathbb{Z}\}$ on $l\subset T_U$. 

We identify the real segment $[q_{j,\alpha},q_{j+1, \alpha}]$ on $l_\alpha$ with the segment  $[q_j,q_{j+1}]$ on $l$, for each $j$ and $\alpha$. 

In the case the translation distance is zero, for each $\alpha$, the  point $q_\alpha$ in $T_0$, fixed by the $\mathbb{Z}$-action, is identified with the point $q$ in  $T_U$. In this case there are infinitely many copies of the fundamental domain of the $\mathbb{Z}$-action on $T_U$, indexed by by $j\in \mathbb{Z}$, attached to the vertex $q$ (as in Figure 7.)

We obtain a real tree $T$  as the resulting identification space: removing a point from the identified axis or point disconnects each of the trees which we are identifying, hence the identified space, so that space is a tree.

Note that the action of $\mathbb{Z}$ extends to the resulting $T$: in the case of positive transverse measure the translations match along the identified axes, and in the case of zero transverse measure,   the generator of $\mathbb{Z}$ acts on the attached copy of $T_U$ by the $\mathbb{Z}$-action that fixes the vertex $q$, and  takes the $j$-th copy of the fundamental domain to the $(j+1)$-th copy, for each $j\in \mathbb{Z}$. 

In either case, the tree $T$ admits an isometric action of  $\pi_1(S \setminus U) = \pi_1(S\setminus p)$; 
and is  the leaf space of a measured foliation $\tilde{F}$ on $\widetilde{S \setminus p}$ that descends to a foliation $F$ on $S$ with a pole-singularity at $p$.

This foliation $F$ can be thought of as obtained by identifying the foliations on $S\setminus U$ and $U$ along their common boundary $\partial U$. (Note that the leaves of either foliation can be assumed to be either all orthogonal to $\partial U$ or all parallel to it, in the case the boundary has zero transverse measure.) 
Since the transverse measures of arcs are determined by the metric tree,  the equivalence class of such a measured foliation is uniquely determined (see the remark following Definition \ref{mf-pole}).

Conversely, given a suitable choice of disk $U$, any foliation in $\mathcal{MF}(n)$ is isotopic to one such that it restricts both to a model foliation in $\mathcal{P}_n$ on $U$, and to a measured foliation on the surface-with-boundary $S\setminus U$.

Let us now identify the parameter space arising from the constructions of these foliations.  By Proposition \ref{mfgb} the space of measured foliations on $S\setminus U$ is homeomorphic to $\mathbb{R}^{6g-6+3}$, 
and by Corollary \ref{hatpn} the space of possible pointed measured foliations on $U$ is homeomorphic to $\mathbb{R}^{n-1}$.
 Recall that one parameter, namely the transverse measure of $\partial U$, must be identical for both measured foliations, and hence the combined parameter space is $\mathbb{R}^{6g-6+n+1}$, as required. \end{proof}

\textit{Example.}  A special class of foliations are those of \textit{half-plane differentials} defined in \cite{GW1}; such a foliation has a connected critical graph that forms a spine of the punctured surface, and whose complement comprises foliated half-planes. Since any essential simple closed curve is homotopic to the critical spine, their transverse measures all vanish; so do the ``local parameters" as in Definition \ref{mf-pole} since there are no strips.

\section{Proof of Theorem \ref{thm1}}

We shall use the technique of harmonic maps to $\mathbb{R}$-trees as introduced in \S2.4.
For ease of notation, we shall deal with the case of a single pole $p$, that is, the cardinality of the set $\mathcal{P}$ of marked points on the closed Riemann surface $\Sigma$  is one; the arguments extend \textit{mutatis mutandis} to the case of several poles.\\

Let $\widetilde{X}$ denote the universal cover of $X= \Sigma \setminus p$, and let $ \widetilde{F}$ be the lift of the measured foliation $[F] \in \mathcal{MF}(n)$ to this universal cover. 
Let $T$ be the $\mathbb{R}$-tree that is the leaf-space of $\widetilde{F}$.

The collapsing map of the foliation $\widetilde{F}$ defines an equivariant map from $\widetilde{X}$ to $T$ that  is harmonic. Conversely, given such an equivariant harmonic map, its Hopf differential (see \S2.4) descends to a holomorphic quadratic differential on $X$.  

The strategy of proof shall be to prove existence and uniqueness statements for such maps, with prescribed asymptotic behaviour at the pole.

\section*{Uniqueness}

\begin{prop}\label{unq} Let $q_1$ and $q_2$ be meromorphic quadratic differentials on $\Sigma$ with a pole of order $n$ at $p$.  If they induce the same measured foliation $[F] \in \mathcal{MF}(n)$, and have identical principal parts with respect to the (already) chosen coordinate $U$ around $p$, then $q_1=q_2$.
\end{prop}
\begin{proof}
In the universal cover, these two quadratic differentials determine  two harmonic collapsing maps ${h_1}$ and ${h_2}$ from $\widetilde{X}$ to the same $\mathbb{R}$-tree $T$.  Recall that distances in the tree correspond to transverse measures of arcs by (\ref{trans}). Let $\gamma_i$ be a path in $\widetilde{X}$ from some basepoint $x_0$ to $x$, that is transverse to the horizontal foliation $\widetilde{F_i}$ for $q_i$. Then

\begin{equation}\label{diff}
d_T({h_1}(x), {h_2}(x)) = \left\vert \displaystyle\int\limits_{\gamma_1} \left\vert\Im(\sqrt {q_1}\right\vert - \int\limits_{\gamma_2}\left\vert\Im\sqrt{q_2})\right\vert \right\vert   
\end{equation}

If the principal parts for $q_1$ and $q_2$ coincide, then examining the expressions in (\ref{princ1}) and (\ref{princ2}) we can conclude that in a neighborhood of the pole,  we have:
\[ 
  \sqrt {q_1} - \sqrt{q_2} =  
\left\{
\!
\begin{aligned}
 g_1(z) - g_2(z) & \text{  if  } n \text{ is even}\\
 z^{-1/2}(g_1(z) - g_2(z))  & \text{  if  } n \text{ is odd}
\end{aligned}
\right.
\] \label{eqn:local diffs near}
where $g_1$ and $g_2$ are  bounded holomorphic functions in such a neighborhood.

There are two consequences of these expressions.  First, note that there is a neighborhood $U$ of the pole so that, for $x \in U$, we may find a common arc $\gamma$ that runs from a point on $\partial U$ to $x \in U$ that is simultaneously transverse to the horizontal foliation $\widetilde{F_1}$ for $q_1$ and to the horizontal foliation $\widetilde{F_2}$ for $q_2$.  Thus choosing $\gamma_1$ and $\gamma_2$ to contain such a common vertical arc $\gamma(x)$ in the neighborhood $U$, we may write 

\begin{align}\label{eqn:diff2}
d_T({h_1}(x), {h_2}(x)) &= \left\vert \displaystyle\int\limits_{\gamma_1} \left\vert\Im(\sqrt {q_1})\right\vert - \int\limits_{\gamma_2}\left\vert\Im(\sqrt{q_2})\right\vert \right\vert \notag \\
&=C_0 + \left\vert \displaystyle\int\limits_{\gamma(x)} \left\vert\Im(\sqrt {q_1})\right\vert - \int\limits_{\gamma(x)}\left\vert\Im(\sqrt{q_2})\right\vert \right\vert \notag \\
&= C_0 +  \displaystyle\int\limits_{\gamma(x)} \left\vert \Im(\sqrt {q_1} - \sqrt{q_2}) \right\vert    
\end{align}
where $C_0$ bounds the expressions from \eqref{diff} from the portion of the curves $\gamma_1$ and $\gamma_2$ in the complement of $U$; the second equality follows after combining the integrals over the common domain curve $\gamma(x)$ near which we can simultaneously orient both foliations $\widetilde{F_1}$ and $\widetilde{F_2}$.

The second consequence of the principal parts being identical and the expression \eqref{eqn:local diffs near} holding in a neighborhood $U$ is
that the (restricted) integral in \eqref{eqn:diff2} over $\gamma$ is bounded. In the case that $n$ is even, this is immediate as the integrand is bounded (by a constant $C$), and when $n$ is odd, the integrand is integrable since in the neighborhood we have:
\begin{equation*}
\displaystyle\int\limits_{\gamma(x)} \left\vert \Im(\sqrt {q_1} - \sqrt{q_2}) \right\vert \leq \left\vert \displaystyle\int\limits_\gamma  \frac{C}{z^{1/2}} dz \right\vert   = O(1) 
\end{equation*}
We conclude from this estimate and \eqref{eqn:diff2} that $d_T({h_1}(x), {h_2}(x))$ is bounded independently  of $x$.

Since this distance function is also subharmonic (see Lemma \ref{dist}), and a punctured Riemann surface is parabolic (in the potential-theoretic sense), then we see that the distance function is constant. An argument identical to the one in the proof of Proposition 3.1 in \cite{DaDoWen} then shows that this constant is in fact zero. We briefly recount the argument of theirs:  suppose $x\in \mathbb{H}^2$ is a zero of $\tilde{q}_1$ with a neighborhood $\Delta$ that does not include any zeroes of $\tilde{q}_2$, except perhaps $x$ itself. Then the collapsing map ${h_1}$ is constant on a horizontal arc $e \subset \Delta$ that has an endpoint $x$. Since the image of the arc ${h_2}(e)$ is at a constant distance  from ${h_1}$, and a sphere in a tree is discrete, the image must be a point. Hence $e$ is also horizontal for the quadratic differential $\tilde{q}_2$; since this is true for any choice of $e$, and the level set $h_1^{-1}(h_1(e))$ branches at $x$, then the level set $h_2^{-1}(h_2(e))$ also branches at $x$. This implies that $x$ is also a zero of $q_2$. Thus, the zeroes of $\tilde{q}_1$ and $\tilde{q}_2$ coincide, and since it is a basic consequence of the agreement of the principal parts that the poles have the same order, we see that $q_1$ and $q_2$ are constant multiples of each other. Since the principal parts are identical, the two differentials are in fact equal. 
\end{proof}

\section*{Existence}

In this section we prove that there exists a meromorphic quadratic differential on $\Sigma$ with a pole of order $n$ at $p$ that induces a given class of a measured foliation $[F] \in \mathcal{MF}(n)$, and has any prescribed principal part $P$, as long as $P$ is compatible with $F$. (For the notion of compatibility, see Definition~\ref{compat}.)  \\

Choose a representative foliation $F^\prime \in [F]$ such that $F^\prime \vert_U$ is identical to the foliation induced by $P^2dz^2$ on $U$. By a suitable isotopy of the foliation, we can further assume that  $F^\prime \vert_U$ is a ``model foliation" in $\mathcal{P}_n$ (see Definition \ref{defn:pn}). \\

Recall that $X = \Sigma \setminus p$.
Lifting to the universal cover, the collapsing map for the foliation $\widetilde{F^\prime}$ defines an equivariant map 
\begin{equation}\label{coll-phi}
\phi:\widetilde{X} \to T
\end{equation}
such that on any lift $\widetilde{U\setminus p}$, the map  $\phi$ coincides with the (lift of the) collapsing map for the meromorphic quadratic differential $P^2dz^2$.
 (See \S2.2 for the notion of a collapsing map.)  \\

Note that since the foliation $F^\prime$ on $U$ is in the space of model foliations $\mathcal{P}_n$,  the leaf space of the lifted foliation $\widetilde{F^\prime}$ when restricted to any lift  $\widetilde{U\setminus p}$ is a tree $T_U$ with a $\mathbb{Z}$-action by isometries, as described in \S3.2. 

The collapsing map $\phi$ restricted to any such lift thus maps to the subtree $T_U$. \\

By a method of exhaustion and taking limits, we shall produce a $\pi_1(X)$-equivariant  harmonic map 
\begin{center}
$h:\widetilde{X} \to T$ 
\end{center}
such that the distance between the two maps $h$ and $\phi$  is asymptotically bounded. \\

We begin by considering an exhaustion 
\begin{equation}\label{exh}
X_0 \subset X_1 \subset X_2 \subset \cdots 
\end{equation}
of the punctured surface $X$ by compact subsurfaces with boundary (a more particular exhaustion shall be chosen in \S4.2).

\subsection{Harmonic map from surface-with-boundary} \label{sec: Harmonic map from surface-with-boundary} For each $i\geq 1$, let $T_i$ be the truncation of $T$ that is the leaf-space $F^\prime$ restricted to the lift of the subsurface $X_i$; here we imagine $T_i$ as a truncation of the tree $T$. 
For each $i$, we first solve the Dirichlet problem of finding an equivariant harmonic map $h_i:\widetilde{X_i} \to T_i$ from $\widetilde{X_i}$ to $T_i$  that:

(I)  agrees with the map $\phi$ on the lifts of $\partial X_i$ and

(II)  collapses along a foliation that is measure-equivalent to the restriction of $\widetilde{F^\prime}$ on $\widetilde{X_i}$ . \\

 This harmonic map $h_i$ is obtained by:\\

(a)  Taking a limit of a sequence of maps $h_i^m$ (where $m\geq 1$) from  $\widetilde{X_i}$ to $T_i$ with the prescribed boundary condition (I) on   $\phi\vert_{\widetilde{\partial X_i}}$ whose energy tends to the infimum of energies for all such maps. The subsurface being compact provides an energy bound, and Lemma~\ref{cour} guarantees equicontinuity.

We wish to then apply the Ascoli-Arzela theorem: see the remark just after Lemma~\ref{cour}. To do so, we need, for each point $x \in \tilde{X_i}$, to trap the images $h_i^m(x)$ in a compact set within the non-locally compact tree. Here we cite the proof, in the case of a closed surface found in \cite{Wolf2}, Lemmas~3.3 and 3.4; an identical argument works for a compact surface with boundary (as in the present case).

A brief summary of that proof is the following: Elements of the fundamental group $\pi_1(S)$ act on the tree by isometries, translating along different isometric copies of $\R$, called axes. The equicontinuity of the maps $h_i^m$, together with the fact that the axes of elements only meet along compact subintervals before diverging (reflecting the NPC nature of the tree), result in a uniform bound on the distance between the maps $h_i^m$ and $h_i^M$ for any $M>m$.

We then look at the image $h_i^m(\tilde{B})$ of the lift $\tilde{B} \subset \tilde{X_i}$ of a non-trivial closed curve $B$ on $X_i$. This image must meet the axis of the isometry represented by $[B] \in \pi_1(S)$ on the tree $T$. By the uniform distance bound we eventually find a point, say $x \in \tilde{X_i}$, whose image $h_i^m(x)$ under $h_i^m$ converges.

The rest of the argument slowly leverages the existence of simpler sets of points whose images converge or lie in a compact set into larger sets of points with those properties. Start from that one point $x \in \tilde{X_i}$ and then connect that point $x$ with itself on $\tilde{X_i}$ along the lift of a closed curve $B$ that goes through a singularity $z_i^m$ of the lifted  foliation: we then trap the images $h_i^m(z_i^m)$ of that singularity $z_i^m$ into a segment along an axis corresponding to $B$, eventually finding that limit points $z_i$ of singularities have the boundedness property we need.  This argument is then promoted to arcs between such limiting singular points $z_i$ and then finally to cells that these arcs bound.  

The only comment needed to explicitly adapt this proof to the present setting is to note that the boundary arcs $\partial X_i$ already have images lying in a compact set of the tree $T$, as their images were fixed by the hypotheses of the problem.

(b) It remains to check that the resulting harmonic map $h_i$ to $T$  has a Hopf differential with a foliation $\mathcal{F}_i$ that is measure equivalent to the one desired. Let $T^\prime_i$ be the leaf-space of the foliation $\mathcal{F}_i$. It is well-known (see, for example, Proposition 2.4 of \cite{DaDoWen}) that there is a morphism of $\mathbb{R}$-trees $\Psi_i: T^\prime _i \to T_i$ such that the harmonic map $h_i:\widetilde{X_i} \to T_i$ factors through the collapsing map $c_i$ of the foliation $\mathcal{F}_i$, that is, we have $h_i =  \Psi_i \circ c_i$. Moreover, the morphism $\Psi_i$ is $\pi_1(X)$-equivariant where the surface-group action on $T_i$ (and $T^\prime_i$) is \textit{small}, that is, stabilizers of arcs are at most infinite-cyclic. However the analogue of Skora's theorem (\cite{Skora}, see also \cite{FarWol}) for a surface-with-punctures then asserts that any such morphism must in fact be an isometry (\textit{i.e.} there is no ``folding").

These steps (a) and (b) complete the proof of the existence of an equivariant harmonic map $h_i$ satisfying (I) and (II) above. 

For later use, we note the uniqueness of the solution of this equivariant Dirichlet problem.

\begin{prop} \label{prop: h-i unique}
	The solution $h_i$ to the equivariant Dirichlet problem (I) and (II) is unique.
\end{prop}

\begin{proof}
	The style of argument is standard (\cite{Schoen}, see also \cite{Mese}):  note that in the prescribed homotopy class, we may construct an equivariant homotopy along geodesics connecting the pair of image points $h_i(p)$ and $h_i'(p)$, and along that homotopy, the (equivariant) energy is finite, convex and critical at the endpoints. We conclude that the energy is then constant along this homotopy and, by examining the integrand for the energy functional, that the maps must differ by an isometry that translates along an axis of the target tree.  However, since the maps agree on the lifts of the boundary, $\partial X_i$, the maps in fact agree.  
\end{proof}

\subsection{Taking a limit}

To show that the family $\{h_i\}$ subconverges to a harmonic map $h:\tilde{X} \to T$, we need to prove a uniform bound on energy of the restrictions to a lift of any compact subsurface.  For this, we shall follow a strategy similar to that in our previous paper \cite{GW1}.
The extension of that argument depends crucially on our ability to choose an exhaustion of the surface that we now describe. 

\subsection*{A symmetric exhaustion}
We shall choose a specific exhaustion of the punctured surface. In the neighborhood $U$ about the puncture, this relies on the \textit{normal forms} (\ref{nform1}) and (\ref{nform2})  for a meromorphic quadratic differential with a pole of order $n$.

We consider the exhaustion of a neighborhood of the puncture by concentric disks in the ``normal" coordinate $z$ of (\ref{nform1}) or (\ref{nform2}). Namely, we have
\begin{equation}\label{exhpole}
U_i = \{ z \in \mathbb{D}^\ast \text{ } \vert \text{ } \lvert z \rvert < \delta/i\} 
\end{equation}
for each $i\geq 1$,  where $\delta>0$ is chosen small enough such that $U_1$ (and hence all subsequent $U_i$) lie in the coordinate disk $U$ chosen earlier.

We thus obtain an exhaustion 
\begin{equation*}
 X \setminus U = X_0 \subset X_1 \subset X_2 \subset  \cdots 
\end{equation*}
where  $X_i = X \setminus U_i$ for each $i\geq 1$.\\

\subsection*{Consequences of symmetry}

Recall that the collapsing map $\phi$ along leaves of the measured foliation for $q$ produces a $\mathbb{Z}$-equivariant map from the universal cover of $U \setminus p$ to the dual tree $T_U$, which is the leaf-space of the lift of the foliation $\mathcal{F}^\prime \vert_U \in \mathcal{P}_n$ .  (See Figure 8.) 

Denote the annulus $A_i = X_{i+1} \setminus X_{1}$.

Since the foliation on $A_i$ for $i\geq 1$ is a restriction of the model foliation in $\mathcal{P}_n$,  the image of the restriction  $\phi_i = \phi\vert_{\tilde{A}_i}$ is the tree $T_U$ (see \S3.1). 

Passing to the quotient by the $\mathbb{Z}$-action, we obtain a map 
\begin{equation*}
\bar{\phi_i} :A_i \to \overline{T}_U 
\end{equation*}

where $\overline{T}_U$ is a graph with single cycle (the quotient of $T_U$ by the $\mathbb{Z}$-action) - see, for example, the left figure in Figure 4. \\

An immediate consequence of the preceding construction of an exhaustion  is:

\begin{lem} When $A_i$ is uniformized to a round annulus, the map $\bar{\phi}_i$ above has an $n$-fold rotational symmetry when $n$ is odd, and an $n/2$-fold rotational symmetry, when $n$ is even.
\end{lem}
\begin{proof}
The uniformized round annulus can be taken to be
\begin{equation*}
A_i = \left\{ z \in \mathbb{D}^\ast \text{ } \vert \text{ }  \frac{\delta}{i} <  \lvert z \rvert < \delta\right\} 
\end{equation*}
The desired rotational symmetry is  evident from the fact that differential in (\ref{nform1}) is invariant under the coordinate change $z\mapsto \displaystyle e^{\frac{2\pi i}{n}} z$, for $n$ odd, and the differential in (\ref{nform2}) is invariant under the coordinate change $z\mapsto \displaystyle e^{\frac{2\pi i}{n/2}} z$ for $n$ even.
\end{proof}

In particular, the corresponding image tree $T_U$ has an $n$- (resp. $n/2$-) fold symmetry when $n$ is odd (resp. even).

The key advantage of the symmetry, that the following lemma asserts,  is that the map $\phi_i$ can be further thought of as a branched cover of a harmonic map from a ``quotient" annulus $\bar{A}_i$  to an interval.\\

In what follows we shall fix $n$ to be even.

\begin{lem}\label{mz}
In the setting just described, there is a map  $f_i$ from an annulus $\bar{A_i}$ to an interval $[-c_i,c_i]$ for some $c_i>0$  such that
\begin{itemize}
\item $f_i$ has mean zero, and
\item the map $\bar{\phi}_i$ on $A_i$  is an $n/2$-fold cover of $f_i$. 
\end{itemize}
Namely, there is a $n/2$-fold covering map  $p_i: A_i \to \bar{A_i}$ and a branched covering $b_i$ from $\overline{T}_0$ to $[-c_i,c_i]$ branched over $0$, such that $f_i \circ p_i = b_i \circ \bar{\phi}_i$.
\end{lem}
\begin{proof}

Note that by the change of coordinates $z\mapsto z^{n/2}=w$,  the differential (in a neighborhood of the pole) is seen to be the pullback by an $n/2$-fold branched cover of the differential
\begin{equation}\label{nform4}
\frac{4}{n^2}\left(\frac{1}{w^2} + \frac{a}{w}\right)^2 dw^2 
\end{equation}
where the branching is over $0$ in the $w$-coordinates.

Consider the exhaustion of  a neighborhood of such a pole by 
\begin{equation}\label{baru}
\bar{U}_i = \{ w \in \mathbb{D}^\ast \text{ } \vert \text{ } \lvert w \rvert < (\delta/i)^{n/2}\} 
\end{equation}
which lifts to our exhaustion (\ref{exhpole}) by the $n/2$-fold branched covering $z\mapsto z^{n/2} =w$.

We can then define the annulus $\bar{A_i} := \bar{U}_1 \setminus \bar{U}_{i}$ and the map $f_i$ to be the collapsing map for the differential in (\ref{nform4}). 

This map has mean zero, as one can easily verify: the dual metric tree to the induced foliation on the $w$-plane is the real line $\mathbb{R}$,  and the collapsing map  has the expression 
\begin{equation}\label{expr-coll}
f_i(w)  = \Im\left(\frac{1}{w^2} + \frac{a}{w}\right)
\end{equation}
where for convenience we have dropped the multiplicative real factor as in (\ref{nform4}).
In polar coodinates, this can be written as:
\begin{equation}\label{expr-coll2}
f_i(r,\theta)  = -\frac{\sin2\theta}{r^2} + \frac{\Im a}{r}\cos \theta - \frac{\Re a}{r}\sin \theta
\end{equation}
which has mean zero on any circle centered at the origin ($r= \text{constant}$).

Note that the image of its restriction to the boundary $\partial \bar{U}_i$ is the interval $[-c_i,c_i]$ where $c_i$ is the maximum value of the restriction.

Since the original differential given by the normal form (\ref{nform2})  is an $n/2$-fold branched covering of (\ref{nform4}), the collapsing map $\bar{\phi}_i$ is  $n/2$-fold branched cover of $f_i$. 
\end{proof}

\subsection*{Decay along annulus}
The key conclusion in the previous lemma is that $\bar{\phi}_i$ is then the  lift of a harmonic \textit{function} on a cylinder \textit{with mean zero} on any meridional circle. The following technical lemma (a strengthened version of Proposition 4.11 of \cite{GW1}) proves an ``exponential decay" for any harmonic function which has such mean-zero boundary conditions.  This will be subsequently used to prove a uniform energy bound on compacta for the harmonic maps $h_i$.

\begin{lem}\label{decay}
Let $C(L)$ be a flat Euclidean cylinder of circumference $1$ and length $L>1$. 
 Let $h:C(L)\to \mathbb{R}$ be a harmonic function with identical maps $f:S^1 \to \mathbb{R}$ on each boundary that satisfy:
\begin{itemize}
\item the maximum value of $\lvert f \rvert $ is $M$, and
\item the average value of $f$ on each boundary circle is $0$.
\end{itemize}
Then the maximum value of the restriction of $h$ to a fixed collar neighborhood of  the central circle is bounded by $O(Me^{-L/2})$, i.e there is a universal constant $K_0$ so that $\lvert h(L/2,\theta) \rvert  \leq   K_0 M e^{-L/2}$, independent of the boundary values $f$ of $h$.

Moreover, such an exponential decay holds for the derivative $\partial_\theta h$.
\end{lem}

The proof is by a straightforward ``spectral decay" argument and is an extension of a similar derivation in \cite{GW1}; for the sake of completeness, we include a proof, but so as not to interrupt the discussion, we relegate that argument to Appendix~\ref{appendix:decay}.

\begin{figure}
  \centering
  \includegraphics[scale=0.4]{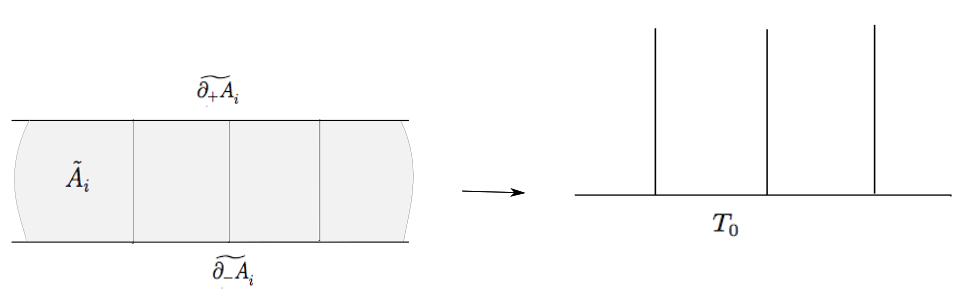}\\
  \caption{The solution of the partially free boundary problem is the $\mathbb{Z}$-equivariant map of least energy  from $ \tilde{A}_i$ to $ T_U$ that restricts to $\phi\vert_{\widetilde{\partial_+ A}_i.}$ on one boundary component. }
  \end{figure}

\subsection*{Partially free boundary problem} The key to proving the uniform energy bound is to control the harmonic map (\textit{ie.} establish $C^0$-bounds) in any lift of the annulus ${A}_i$.  In what follows it shall be useful to consider the harmonic map
\begin{equation}\label{psii}
 \psi_i: \tilde{A}_i \to T_U
 \end{equation}
  that is equivariant with respect to the $\mathbb{Z}$-action on the spaces, and  solves the following

\begin{prob} \label{problem: PFBVP}
	(The partially free boundary value problem.) Find a map $\psi_i: \tilde{A}_i \to T_U
$ which minimizes the equivariant energy amongst  $\mathbb{Z}$-equivariant locally square-integrable maps from $\tilde{A}_i$ to $T_U$ that restrict to the map $\phi\vert_{\widetilde{\partial_+{A_i}}}$ on the lift of one of the boundary components $\partial_+A_i$, but has no prescribed condition on the  lift of the other boundary component $\partial_{-}A_i$.
\end{prob}
 
(Recall that $A_i = X_{i+1} \setminus X_1$. Here and in what follows ${\partial_+A}_i$ shall mean $\partial X_{i+1}$ and $\partial_-A_i$ that we call the ``free boundary" will mean $\partial X_1$.) \\

We shall estimate the map $\psi_i$, and along the way, also prove the existence of such a map. We summarize the latter as the following proposition.

\begin{prop}\label{doub}
Consider the annulus $\hat{A_i}$ obtained by doubling across the free boundary and solve the Dirichlet problem on the bi-infinite strip $\tilde{A}_i$ that is its universal cover, with identical (and periodic) boundary conditions determined by the model map $\phi$ on each boundary component. Call this map $\hat{\psi_i}$. 
Then the restriction of $\hat{\psi_i}$ to one of the halves of the strip (i.e. one that is a lift of $\tilde{A}_i$) is a solution $\psi_i$ of the partially free boundary problem (Problem \ref{problem: PFBVP}). 
\end{prop}

\emph{Sketch of proof}[details in the Appendix]: We provide a brief sketch of the argument here: We first show that a solution of the partially free boundary problem is characterized by having normal derivatives that vanish at  the points of the free boundary where its image avoids a vertex.  By an analysis of the possible preimages of the vertices of the target tree, we show that such preimages intersect the free boundary at finitely many points (up to the $\mathbb{Z}$-action). These imply that a solution $\psi_i$ has normal derivatives vanishing in all but finitely many points (up to the $\mathbb{Z}$-action), and hence one can define a weakly-differentiable map $\psi_i^*$ on  the universal cover of $\hat{A_i}$ that restricts to $\psi_i$ and its reflection on each half. By the boundedness of the derivatives in the $L^1$-norm, the above points are removable singularities for the holomorphic Hopf differential of $\psi_i^*$.  Hence this map $\psi_i^*$ is harmonic, and  by the uniqueness of the solution to the Dirichlet problem, it must coincide with the map $\hat{\psi_i}$ on the lift of the doubled annulus.

Moreover, by the existence of a solution to this Dirichlet problem (see \S2.2 of \cite{KorSch1}; see also \cite{Wolf2}, and the argument described in section~\ref{sec: Harmonic map from surface-with-boundary}, paragraph (a)), we obtain a solution to Problem~\ref{problem: PFBVP}.
\qed\\

Note that by the uniqueness of solutions to the Dirichlet problem, we then immediately also have:

\begin{cor}
The solution to the partially free boundary problem is unique. 
\end{cor}

\subsection*{A uniform control} Before proving the energy bounds for $h_i$, we continue with our analysis of the solutions to the partially free boundary problem $\psi_i$ that we just introduced; in particular, we aim to control the image of this map on the lift of the free boundary.

Recall that $T_U$  is the leaf-space of the lift of the model foliation in $U \setminus p$ to its universal cover.\\

The preceding lemma is crucial in the following

\begin{prop}\label{ubound}
The map $\psi_i\vert_{\widetilde{\partial_-A_i}}$, which is the solution of the partially free boundary problem when restricted to the lift of the free boundary, has uniformly bounded image in the metric tree $T_U$. 
\end{prop}
\begin{proof}

First, consider the case that $n$ is even. 

Recall from Proposition~\ref{doub} that $\psi_i$ is ``half" of the solution $\hat{\psi_i}$ to a Dirichlet problem on the strip that is the lift of the doubled annulus $\hat{A_i}$, where the identical boundary conditions are determined by the model map $\phi$ on each boundary component).

Note that by the structure of the foliation near the puncture, and its collapsing map $\phi$,  the image of the map $\hat{\psi_i}$ on the  lifts of the two  distinct boundaries of the annulus is the same sub-tree $T_U^i \subset T_U$.

The convex hull of the sub-tree $T^i_U$  is itself, and hence the entire image of the $\mathbb{Z}$-equivariant harmonic map is contained in $T_U$. (Recall that because there is a distance-decreasing map on the (NPC) tree into a convex set, the image of an energy-minimizing map is contained in the convex hull of the image of its boundary, lest a composition with the distance decreasing projection reduce the energy of the map.) 
 
By the equivariance, the solution $\hat{\psi_i}$ passes to a quotient map from $A_i$ to a graph $\overline{T_U}$ with one cycle. Moreover, by the symmetry of the boundary conditions, this quotient map is in fact a $n/2$-fold cover of a harmonic function  $ \bar{\psi_i}$ of mean zero, as in Lemma \ref{mz}. 

As asserted in Lemma \ref{decay}, such a harmonic function with mean zero has an exponential decay to the central circle. 

Namely, consider the coordinates $\{(x,\theta) \vert 0<x<2L, 0\leq  \theta <2\pi\}$ on the double of the annulus $\bar{A}_i$. Then we have:
\begin{equation}\label{bd1}
\lvert \bar{\psi_i} (L,\cdot) \rvert  = O( c_i e^{-L})  
\end{equation}
where note that $L$ equals the modulus of the annulus $\bar{A}_i$, and $c_i$ is the maximum value achieved on the boundary, as in Lemma \ref{mz}. 

This exponential decay is inherited by the map $\hat{\psi_i}$ in the cover. 

Note that the above estimates of exponential decay depends on the modulus of the annulus; this shall be the key estimate on the geometry of the map and its domain annulus we will need, to conclude the proof of the proposition.

 That is, though the maximum boundary value $c_i \to \infty$ gets larger as $i\to \infty$, so does the  modulus of $\bar{A}_i$,  and the above decay shall  balance out to prove the result.

To be more precise: $c_i$ is the maximum value of $\bar{\phi}_i$ given by (\ref{expr-coll}), when the map is restricted to the circle $\lvert w \rvert = (\delta/i)^{n/2}$. Hence  it grows like $O(i^n)$. On the other hand, the modulus of $\bar{A}_i$ can be calculated to be, from (\ref{baru}), of the order of $O(n\ln i)$. Substituting in (\ref{bd1}), we obtain $\lvert \bar{\psi_i} (L,\cdot) \rvert = O(1)$ as desired. 

All the preceding discussion in this proof was for the case when $n$ was even. 

For $n$ odd, we reduce it to the former case by a trick of passing to a further double cover:

Namely, consider the quadratic differential with a pole of order $2n$ obtained by a two-fold branched cover of (\ref{nform1}). The exhaustion (\ref{exhpole}) lifts to a neighborhood of this pole, and  one can consider the partially free boundary problem on the lifts of the corresponding annuli $\hat{A}_i$, each a double cover of the original annulus $A_i$. 
The solution $\psi_i$ to the partially free boundary problem on $\tilde{A}_i$ is then a two-fold quotient of the solution $\hat{\psi_i}$ on this double cover. By the preceding argument for an even-order pole, the map $\hat{\psi_i}$  is uniformly bounded on the free boundary  and hence so is the quotient map $\psi_i$.
\end{proof}

We record for future use the following sharper consequence of the previous argument:

\begin{cor}\label{cor-20}
There exists a $\delta>0$ such that  $\psi_i\vert_{\tilde{A}_\delta}$ has a uniformly (independent of the index $i$) bounded image in $T_U$, where $A_i^\delta$ is a fixed $\delta$-collar of the free boundary component  $\partial_- A_i$ of the annulus $A_i$. 
Moreover, the weak angular derivatives of $\psi_i$ are uniformly bounded on this lift of $A_i^\delta$.
\end{cor}
\begin{proof}
As before, we uniformize $A_i$ as a flat Euclidean cylinder, a fix a lift which is a flat strip $\tilde{A}_i$. 
The decay of $\psi_i$ and $\partial_\theta \psi_i$ towards the lift of the free boundary component $\partial_-A_i=\partial X_1$, implied by Lemma \ref{decay}, provides a uniform bound in any fixed collar-neighborhood of the free boundary. 
\end{proof}

\subsection*{Uniform energy bounds}

We first note the following energy-minimizing property of the map $\phi_i = \phi\vert_{\tilde{A}_i}$, being the collapsing map for the foliation induced by the meromorphic quadratic differential $P^2dz^2$:

\begin{lem}\label{emin} The map $\phi_i$ has minimum equivariant energy among all $\mathbb{Z}$-equivariant smooth maps with the same boundary conditions on the lifts of $\partial_+A_i$ and $\partial_-A_i$.
\end{lem}
\begin{proof}
This follows from the fact $\phi_i$ is a harmonic map to a tree: such a harmonic map with fixed boundary conditions  is unique since the target is negatively curved (see \cite{Mese}). On the other hand, an energy-minimizer exists for the Dirichlet problem (see, for example, the proof sketch for Proposition~\ref{doub}) and is harmonic, hence $\phi_i$ must be the energy-minimizer.
\end{proof}

As in \cite{GW1}, we can now show that Proposition \ref{ubound}  implies that the energy of the solution to the partially free boundary problem is comparable to the restriction of the collapsing map on $A_i$:

\begin{lem}[Energy bounds] \label{lemma:annulus energy comparison}Let $\psi_i$ be the solution to the partially free boundary problem  (\ref{psii}) and $\phi$ the collapsing map in (\ref{coll-phi}). Then we have the following estimate of equivariant energy:
\begin{equation}\label{ebound1}
\mathcal{E}(\psi_i) \leq \mathcal{E}({\phi}\vert_{\tilde{A}_i})  \leq \mathcal{E}(\psi_i) + K
\end{equation}
where $K$ is independent of $i$. 
\end{lem}

\begin{proof}

The first inequality follows from the energy-minimizing property of $\psi_i$. The second inequality follows by noting that \\
(a) the collapsing map $\phi\vert_{\tilde{A}_i}$  solves its own energy-minimizing problem given its boundary values (Lemma \ref{emin}), and\\
(b) we can construct a candidate map for this minimizing problem by adjusting $\tilde{\psi}_i$ (at a uniformly bounded cost of energy) such that the boundary values on the lift of the free boundary $\partial_- A_i$  agree with those of $\phi$.

We provide some more details of (b):\\
Note that the image of $\tilde{\psi}_i\vert_{\widetilde{\partial_-A}_i}$ is uniformly bounded by Proposition \ref{ubound},  and consequently also a uniformly bounded distance from the image of $\phi\vert _{\widetilde{\partial X_1}}$, which is a fixed compact set up to the $\mathbb{Z}$-action. 

Moreover, Corollary \ref{cor-20} asserts that the restriction of $\tilde{\psi}_i$ to the lift of a  a $\delta$-collar neighbourhood of the free boundary $\partial_-A_i$, has uniform bounds on the angular derivative.

The adjustment of the map $\psi_i$ can be then described as follows:\\
Choose the $\delta$- collar neighborhood $A_i^\delta$ of the free boundary $\partial_-A_i$. 
The modification of $\psi_i$ is then by a linear interpolation across the lift of the collar, to achieve the boundary values of $\phi$ at the lift of $\partial_- A_i$. The uniform bound on distance ensures that the stretch in the longitudinal direction is uniformly bounded, and the uniform bound on angular derivatives ensures that so is the stretch in the meridional direction. Hence the interpolating map on the collar has a uniformly bounded energy $K$, independent of the index $i$. 
\end{proof}

This gives, in particular, a lower bound to the equivariant energy of $h_i$ restricted to a  lift of $A_i$:

\begin{cor} The harmonic map $h_i$ defined in \S4.1 satisfies 
\begin{equation}\label{lbd}
\mathcal{E}(h_i\vert_{\tilde{A}_i})  \geq  \mathcal{E}(\psi_i) \geq \mathcal{E}(\phi\vert_{\tilde{A}_i}) - K
\end{equation}
where $\psi_i$ and $\phi$ are as in the above Lemma. 
\end{cor}
\begin{proof}
The first inequality is from the energy-minimizing property of the solution to the ``partially-free-boundary problem" $\psi_i$, and the second inequality is from the second inequality in (\ref{ebound1}).
\end{proof}

The technique of \cite{Wolf3} then provides a uniform energy bound of $h_i$ to compact sets :

\begin{lem}\label{lemma: energy upper bound} For any compact set $Z \subset X$, the restriction of $h_i$ to a lift of $Z$ satisfies 
\begin{equation*}
\mathcal{E}(h_i\vert_{\tilde{Z}}) \leq  C
\end{equation*}
where $0<C<\infty$ depends only on $Z$ (and is independent of $i$). 
\end{lem}

\begin{proof}
Consider the candidate map $g_i$ for the energy-minimizing problem that $h_i$ solves, which restricts to $h_1$ on  the lifts of $X_1$, and to the collapsing map $\phi$ on the lifts of  $A_i$. Note that by construction, these two maps $h_i$ and $\phi$ agree on the lifts of the  common boundary $\partial X_1$; their derivatives may fail to match, but since the curve $\partial X_1$ is real-analytic, the measure of this set of non-differentiability is zero so the map $g$ has locally square-integrable weak derivatives. 

Then for all $i\geq 1$, we have
\begin{equation*}
\mathcal{E}(h_i\vert_{\tilde{X}_1})  + \mathcal{E}(h_i\vert_{\tilde{A}_i})  =  \mathcal{E}(h_i) \leq  \mathcal{E}(g_i)  =  \mathcal{E}(h_1)  + \mathcal{E}(\phi\vert_{\tilde{A}_i})
\end{equation*}

Combining with (\ref{lbd})  then yields 
\begin{equation*}
\mathcal{E}(h_i\vert_{\tilde{X}_1}) \leq K + \mathcal{E}(h_1)
\end{equation*}
where the right-hand-side is independent of $i$.

A similar argument then yields uniform energy bounds on the lift of \textit{any} compact subset $Z$ of $X$: namely,  in the preceding argument, we replace the compact subsurface $X_1$ by a compact subsurface $X_m$ (for some $m>1$) in the exhaustion that contains $Z$. 
\end{proof}

We are now equipped to prove:

\begin{prop} \label{prop: h_i convergence}  
The harmonic maps $h_i:\tilde{X}_i \to T_i$ subconverge, uniformly on compacta, to a harmonic map $h:\tilde{X}\to T$.
\end{prop}

We prove this in two steps: first, we show that the uniform energy bound of  Lemma \ref{lemma: energy upper bound} implies that the \textit{Hopf differentials} $\Phi_i$ of the harmonic maps $h_i$ sub-converge on compacta.  The harmonic maps $h_i$ are then collapsing maps of the foliations induced by $\Phi_i$;  along the convergent subsequence we have control on their leaf-spaces  $T_i$ that are sub-trees of the fixed $\mathbb{R}$-tree $T$. Since $T$ is not locally compact, we need a further topological argument to show that in fact, the images of a fundamental domain under $h_i$ lie in a \textit{compact} subset of $T$. This last fact then ensures the sub-convergence of the sequence $h_i$ by an application of the Arzela-Ascoli theorem.

\vskip.2in

 In the sequel, we index our subsequence of maps as if they were a sequence, simply to avoid typographical complexity. Let $\Phi_i= \Hopf(h_i)$ denote the Hopf differential of the map $h_i$.

\begin{lem} \label{lemma: Hopf diffs converge}
The Hopf differentials $\Phi_i$ subconverge, uniformly on compacta, to a holomorphic differential $\Phi$ on $X$.
\end{lem}
 \begin{proof}
 Choose a compact set $Z \subset X$.  Then Lemma~\ref{lemma: energy upper bound} implies that the total energy $\mathcal{E}(h_i) < C(Z)$ for some constant $C(Z)$ depending on $Z$. But the $L^1$ norm $\|\Phi_i\vert_Z\|_{L^1} = \mathcal{E}(h_i\vert_Z) < C(Z)$, and so the restriction  $\Phi_i\vert_Z$ of the Hopf differentials $\Phi_i$ to the compact set $Z$ are uniformly bounded in norm.  Since $Z$ has injectivity radius bounded from below, we can find balls around every point in $Z$ of uniform size so that around such a point, there is an annulus $A(r, R)$ of inner radius $r \geq \delta$ bounded away from zero on which $\Phi_i$ has uniformly bounded $L^1$ norm.  By Fubini's theorem, we then find, for each point $z \in Z$, a circle $\mathcal{C}$ around $z$ of radius at least $\delta >0$ on which $\int_{\mathcal{C}} |\Phi_i| < C_1(Z)$ is uniformly bounded.  But then the Cauchy integral formula uniformly bounds $|\Phi_i(z)| < C_2(Z)$ and $|\frac{\partial}{\partial z}\Phi_i(z)| < C_3(Z)$.
 
 Thus $\Phi_i\vert_Z$ is a sequence of uniformly bounded holomorphic differentials on a compact set, which subconverge by Ascoli-Arzela. A diagonal argument then gives subconvergence on $X$, as required.
 \end{proof} 

 \begin{lem} \label{lemma: Hopf differential non-trivial}
 The Hopf differential $\Phi$ does not vanish identically on $X$.	
 \end{lem}
 
 \begin{proof}
 If the Hopf differential were to vanish identically, then the approximating differentials $\Phi_i$ would, on each compact set, be uniformly small (for large enough $i$).  However, the distance $\dist_T(h_i(p),h_i(q))$ between image points $h_i(p)$ and 	$h_i(q)$ is bounded below by the horizontal $\Phi_i$-measure of an arc between $p, q \in \tilde{X}$.  
 
 So choose an element $\gamma \in \pi_1(X)$ so that $\gamma$ acts on the tree $T$ by a non-trivial translation along an axis; such an element $\gamma$ is guaranteed by our construction of the tree $T$ as the leaf space of a measured foliation (so that the action of $\pi_1(X)$ on $T$ is small). In particular, we find a $\delta>0$ so that $\dist_T(h_i(p),h_i(\gamma \cdot p)) > \delta$ for any choice of index $i$.  But as $p$ and $q= \gamma \cdot p$ live in some compact set $Z \subset \tilde{X}$, we see from the first paragraph that we may choose $i$ so large that $\dist_T(h_i(p),h_i(q)) < \delta$, a contradiction.
 \end{proof}

\begin{proof}[Proof of Proposition \ref{prop: h_i convergence} ]

By Lemma ~\ref{lemma: Hopf diffs converge},  the Hopf differentials $\Phi_i$ converge to a holomorphic quadratic differential $\Phi$ on the punctured surface $X$. 

By Lemma~\ref{lemma: Hopf differential non-trivial}, the zeroes of such a non-trivial $\Phi$ are isolated, and we see the following: 

For any compact set $Z \subset X$ on which $\Phi$ has no zeroes on $\partial Z$ there is an index $I$ so that for $i>I$, we have that all of the zeroes of $\Phi_i$ lie inside $Z$. 
In addition, the foliations of $\Phi_i$ are uniformly close. Thus we see that, for any fixed fundamental domain $F_Z \subset \widetilde{X \setminus Z}$, the leaf spaces $T_i^Z$ of $\Phi_i\vert_Z$ are all $(1+\epsilon)$-quasi-isometric, in the sense that they are all the same finite topological graph, with edge lengths that are nearly identical.  We will assume that our compact set $Z$ includes the set $X_1$, so that its complement is a subset of the cylinder $X \setminus X_m $ for some $m$.
 
 Now consider a fundamental domain $T^*$ for the action of $\pi_1(X)$ on the tree $T$.  It is also a graph with a finite number of vertices, each of bounded valence, with some infinitely long prongs corresponding to the poles, and otherwise finite length edges.  Because each $\Phi_i$ arises from a solution $h_i$ to an appropriate Dirichlet problem with the horizontal foliation describing the level sets of the maps, each of our trees $T_i^Z$ admits an isometric embedding into $T^*$ (with the truncated prongs being taken into semi-infinite prongs), up to some possible small trimming near the boundary points.\\

 \noindent \emph{Claim. The images of the vertices of $T_i^Z$ in $T^*$ are constant in the index $i$, for $i$ large enough. }
 
 \textit{Proof of Claim.} 
  As a preparatory observation, note that the vertices of the tree $T$ are discrete, in the sense that all maps of continua into the vertex set are constant: this is because the tree $T$ is dual to the measured foliation $[F]$ with which we began this existence proof, and each vertex corresponds to a singular point of the foliation, of which there are but countably many.
 
 To see that the images of the vertices of $T_i^Z$ in $T^*$ are constant in the index $i$, we remark on the construction (section~\ref{sec: Harmonic map from surface-with-boundary}) of the solutions $h_i$ to the equivariant Dirichlet problem on the compact domain $X_i$.  Naturally, since it was formed from a compact exhaustion of $X$, the space $X_i$ may be extended to be a sequence $X_i \subset X_t$ in a continuous family $X_t$ which exhausts $X$; for example, we could take $X_t$ to just be the sublevel set of distance $t$ from some point $p_0 \in X$. Note that by the same proofs for the existence and uniqueness of $h_i$ (again section~\ref{sec: Harmonic map from surface-with-boundary}), we obtain the existence and the uniqueness of a solution $h_t$ to the equivariant Dirichlet problem on $X_t$. 
 
 One can see that these solutions $h_t$ are a continuous family of maps:  Note that each has a holomorphic Hopf differential, and all of the differentials are bounded in $L^1$ on the common domain of definition (say $X_{t_0}$, if $t$ decreases to $t_0$: to see the assertion, we might, for example, just apply the proof of Lemma~\ref{lemma: Hopf diffs converge}). 
  
 Thus, there is a converging subsequence of these differentials that converge, uniformly on compacta.  On the lifts of the boundaries $\partial X_t$, the Dirichlet conditions for the harmonic map $h_t$ (determined by the values of the differentials in a lift of a collar neighborhood)  converge by construction.  Thus the limiting holomorphic differential provides a solution to the equivariant Dirichlet problem for $X_{t_0}$.  But as such a solution is unique by Proposition~\ref{prop: h-i unique}, we see that the family $h_t$ converges to $h_{t_0}$, as desired.
  
 With that continuity established, after one more observation, the rest of the argument will be classical.  Each of the quasi-isometric embeddings $T^Z_t$ defines a finite collection of vertices of the full tree $T$ as images of the vertices in $T^Z_t$.  But we now see that these vertices vary continuously with $t$.  On the other hand, as observed at the start of the argument for the claim, the vertices in a tree form a discrete set, so therefore must be fixed as $t$ varies. This establishes the claim. \qed\\

 We conclude that, for $t$ (or $i$) sufficiently large, the images in $T$ of the zeroes of $\Phi_t$ (or $\Phi_i$) on a fundamental domain $F_Z$ 
 must be constant. Thus, because if the vertices of a subtree are fixed, then so are the edges between those vertices, and we then conclude that the $h_t$- (or $h_i$-) images of $F_Z$ take values in a fixed compact subtree $T_Z$ of the larger not-locally-compact tree $T$.
 
 With that last statement in hand, the rest of the convergence proof is classical:  the maps $h_i$ are equicontinuous on each compacta, which follows from the Courant-Lebesgue Lemma (Lemma~\ref{cour}) applied with the energy bound from  Lemma~\ref{lemma: energy upper bound}. Moreover, the images of any fixed point lies in a compact set. Thus the proposition follows from the Ascoli-Arzela theorem, followed by a diagonal argument.
 
 This concludes the proof of Proposition~\ref{prop: h_i convergence}.
 \end{proof}

\textit{Remark.}
Much of this argument could be replaced by a modification of the (slightly longer) proof of Lemma~3.4 in \cite{Wolf2}, but this argument provides a different (and briefer) explanation.

\subsection{Finishing the proof}

To conclude the proof, we need to verify that the Hopf differential of the limiting map $h$ indeed satisfies the requirements of Theorem \ref{thm1}.

\subsection*{Principal part is $P$}  First, observe that:

\begin{lem} \label{lem:h near phi} The distance $d(h,\phi)\vert_U$ between the restrictions of $h$ and the model map $\phi$  as in \eqref{coll-phi}, to the neighborhood $U$ of the pole, is uniformly bounded.
\end{lem}
\begin{proof} Since $h_i\to h$ uniformly on compact sets, the restriction of the maps $h_i$ to the lifts of the boundary circle $\partial U = \partial X_1$ are of uniformly bounded distance (by say $B>0$) from the corresponding restriction of the map $\phi$. Consider the restriction of these maps to any lift $\tilde{A}_i$ of the annulus $A_i$. By definition, the distance between $h_i$ and $\phi$ on the boundary $\widetilde{\partial_+ A}_i = \widetilde{\partial X}_i$ is zero. By the preceding observation, this distance is uniformly bounded above by $B$ on $\widetilde{\partial_-A}_i$. Since the distance function is $\mathbb{Z}$-equivariant, it defines a subharmonic function on the annulus $A_i$ which is uniformly bounded on the boundary components. Applying the Maximum Principle, we conclude that it is uniformly bounded by $B$ throughout $A_i$. Since this holds for each $i$, we obtain the same bound, on any compact set, for the distance between the limiting map $h$ and $\phi$.
\end{proof}

\begin{cor}\label{HopfP}
 The Hopf differential of $h$ has principal part $P$.
\end{cor}
\begin{proof}
Recall that the Hopf differential of $\phi$ is $P^2dz^2$, and hence has principal part $P$ (as defined in \S2.3). 
The proof then is exactly as in the proof of Proposition \ref{unq}: namely, if $\text{Hopf}(h)$ and $P$ differed at some term (involving a negative power of $z$), then the distance function between the maps to the corresponding trees will blow up nearer the pole. This divergence then contradicts the previous lemma that the distance function between the corresponding collapsing maps is uniformly bounded.
\end{proof}

\subsection*{Measured foliation is $F$} Finally, it  remains to check that

\begin{lem}\label{folF}The measured foliation $\mathcal{F}$  induced by the Hopf differential $\text{Hopf}(h)$ is measure-equivalent to $F$.  

\end{lem}

\begin{proof}  Recall that there is a morphism between the leaf-space $\mathbb{R}$-tree $\mathcal{T}$ for the lifted foliation $\widetilde{\mathcal{F}}$ to the desired $\mathbb{R}$-tree $T$, that the harmonic map $h$ factors through. (See, for example, Prop 2.4 of \cite{DaDoWen}.)  It  suffices to prove that this morphism from $\mathcal{T}$  to $T$ is an isometry. 
 
 We first note that since the harmonic maps $h_i\to h$ uniformly, this $C^0$-convergence can be promoted in regularity to a local $C^1$-convergence: for smooth maps from two-dimensional domains, this is a standard application of elliptic regularity and the Cauchy-integral formula (as in the proof of Lemma~\ref{lemma: Hopf diffs converge}). 
To see this in our more general setting where the target is a tree, note that the difficulty is that not only are the maps $h_i$ not smooth in a classical sense at the zeroes of the Hopf differentials, but also the estimates for the $C^1$ norms can depend on the distance to those zeroes. However, these zeroes are isolated points, and classical regularity theory can be applied on restricting to the regions disjoint from them where the target is locally a segment, where the $C^0$-convergence can be promoted in regularity to a local $C^1$-convergence on those non-singular points. But the Hopf differential of $h_i$ is defined only in terms of the first derivatives of $h_i$, and the global energy bound will bound the $L^1$ norm of $Dh_i$ on some circle in any annular region on the surface with a bound that depends only on the radius of the circle (which we can take to be of moderate size and not dependin on distance to the zeroes of the Hopf differential).  The Cauchy-integral formula for that Hopf differential $\Hopf(h_i)$ then applies and one obtains convergence of the Hopf differentials $\Hopf(h_i)$ at both preimages of vertices as well as at preimages of edge points. Since the measured foliations for $h_i$ are induced by these quadratic differentials $\Hopf(h_i)$, the convergence of the Hopf differentials then implies the convergence of the associated measured foliations.

 Moreover, we have noted in \S4.1 that for each approximating harmonic map $h_i$, the corresponding morphism to the truncation $T_i$ of $T$ is an isometry.

Then to verify that there is no folding from the limiting tree $\mathcal{T}$ to $T$ -- concluding the proof -- we argue as follows.

Choose a pair of distinct leaves $l,l^\prime$ of $\widetilde{\mathcal{F}}$ that map to two distinct points $p,q$ in the dual tree $\mathcal{T}$  for $\widetilde{\mathcal{F}}$. It suffices to show that the morphism from $\mathcal{T}$ to $T$ takes $p$ and $q$ to two distinct points:
Since the foliations $\widetilde{\mathcal{F}_i} \to \widetilde{\mathcal{F}}$ by the argument above, their leaf-spaces converge in the Gromov-Hausdorff sense to  $\mathcal{T}$.  In particular, let the leaves $l_i,l_i^\prime$ of $\widetilde{\mathcal{F}_i}$ converge to $l$ and $l^\prime$ respectively.  Then the distance between the images of the leaves $l_i,l_i^\prime$  in the $\mathbb{R}$-tree for $\widetilde{\mathcal{F}_i}$, is uniformly bounded below, for all large $i$. 
However, since from the previous paragraph there is no folding for these approximating foliations, we see that the images of these leaves by the map $h_i$ have a distance in $T_i \subset T$ that is uniformly bounded below. Since $h_i \to h$ uniformly, this lower bound persists in the limit, and hence $h$ maps these leaves to distinct points of $T$. \end{proof}

\subsection*{Conclusion of the proof of Theorem \ref{thm1}} 
The existence part of Theorem  \ref{thm1} now follows from Corollary \ref{HopfP} and Lemma \ref{folF}, while statement of the uniqueness was proven in Proposition \ref{unq}. \qed

\section{Relation to singular-flat geometry: shearing horizontal strips}

Our Theorem \ref{thm1} asserts that there are local parameters at each pole (namely, coefficients of the principal part with respect to a chosen coordinate chart), that, together with the measured foliation, uniquely specify a meromorphic quadratic differential. Recall that these parameters form a space $\prod_{i=1}^k(\mathbb{R}^{n_i-2}\times S^1)$  of total dimension $N - k$ where  $N = \sum\limits_{i=1}^k n_i$ is the sum of the orders of all poles (combining the contributions from each pole - see Lemma \ref{comp}). We conclude the paper by discussing a geometric viewpoint for these parameters, for the case of a \textit{generic} meromorphic quadratic differential. \\

Let $\mathcal{Q}(S, n_1,n_2,\ldots n_k)$ be the space of meromorphic quadratic differentials on a surface $S$ of genus $g\geq 1$ (with respect to varying complex structures),  with $k\geq 1$ poles of order $n_i\geq 3$ for $1\leq i\leq k$. This total space of quadratic differentials can be considered as a bundle over Teichm\"{u}ller space:
\begin{center}
$\mathcal{Q}(S, n_1,n_2,\ldots n_k)$\\
\vspace{.05in}
$\bigg\downarrow$\\
\vspace{.05in}
$\mathcal{T}_{g,k}$
\end{center}

For any foliation $F \in \mathcal{MF}(n_1,n_2,\ldots n_k)$, let $Q(F) \subset \mathcal{Q}(S, n_1,n_2,\ldots n_k)$ be the subspace of the total space of all differentials whose induced horizontal measured foliation is $F$. 

The main result of the paper, Theorem \ref{thm1}, asserts that the above projection, when restricted to $Q(F)$, is a surjective map to $\mathcal{T}_{g,k}$. 
A dimension count for $Q(F)$ then gives $6g-6+ 2k$ parameters coming from $\mathcal{T}_{g,k}$, and $N-k$ for the fiber over each point as noted above, yielding a total of $\chi = 6g-6 + \sum\limits_i (n_i + 1)$.

In this section, we introduce the operation of ``shearing" along horizontal strips in the foliation, and observe that as a consequence of the work of Bridgeland-Smith in \cite{BriSmi}, for a generic foliation $F$,  this parametrizes  a neighborhood of any point in $Q(F)$ (see Proposition \ref{shear}).

\subsection*{Horizontal strips} The key geometric feature of a generic  differential with higher order poles  is that the induced singular-flat metric has \textit{horizontal strips} as introduced in \S2.3, namely,  maximal subdomains isometric to  $\mathcal{S}(a) = \{ z\in \mathbb{C}| -a < \mathrm{Im}(z) <a\}$ for some $a\in \mathbb{R}_+$, with the induced horizontal foliation being the horizontal lines $\{ \Im z = \text{constant}\}$.

In fact, a generic element of $\mathcal{Q}(S,  n_1,n_2,\ldots n_k)$  has 
\begin{itemize}
\item  all simple zeroes,
\item an induced horizontal foliation with each non-singular leaf starting and ending at poles, and 
\item an induced singular flat metric that comprises $\chi = 6g-6 + \sum\limits_i (n_i + 1)$ horizontal strips, in addition to the $(n_i-2)$ half-planes around each pole. 
\end{itemize}
For this, we refer to the work in \cite{BriSmi}  (see, for example \S4.5 of that paper). A generic differential as described above is a  ``saddle-free GMN-differential", as introduced in that paper. \\

We also recall from \cite{BriSmi} that each horizontal strip $H$ of a generic differential $q$
 has a \textit{unique} zero on each boundary component, and the complex \textit{period} of the horizontal strip  is then
\begin{equation}
\mathrm{Per}(H) =  \pm \displaystyle\int\limits_\gamma \sqrt q
\end{equation}
where $\gamma$ is an arc in the strip between the two zeroes. (Here the sign can be chosen such that $\mathrm{Per}(H) \in \mathbb{H}$, by a choice of `framing' in the orientation double-cover.)

\subsection*{Shears} The absolute value of the imaginary part of such a period gives the \text{width} or transverse measure across the strip, \textit{c.f.} (\ref{trans}). 

The following geometric operation on a horizontal strip keeps the width fixed, but changes the real part of the period.

\begin{defn}  A \textit{shear} along horizontal strip of width $w$  is the operation of  cutting along the bi-infinite horizontal leaf in the middle (at height $w/2$), that is isometric to $\mathbb{R}$, and gluing back by a translation. The \textit{shear parameter} is the real number that is the translation distance, where the sign is determined by a choice of framing or orientation of the strip.   
\end{defn}

Performing a shear with parameter $s\in \mathbb{R}$ along an (oriented) horizontal strip  results in a new singular-flat surface in $\mathcal{Q}(S, n_1,n_2,\ldots n_k)$, that is a new Riemann surface equipped with a meromorphic quadratic differential with the given higher order poles.  
In what follows let $\chi = 6g-6 + \sum\limits_i (n_i + 1)$. Our main observation is:

\begin{prop}\label{shear} Let $F\in \mathcal{MF}(n_1,n_2,\ldots n_k)$ be a generic measured foliation, that is,  having $\chi$ horizontal strips. Let $Q(F)$ be the corresponding subspace of the total space $\mathcal{Q}(S, n_1,n_2,\ldots n_k)$. Then  for any  singular-flat surface $\Sigma_0 \in Q(F)$, there is a neighborhood $\mathcal{V}$ of the origin in $\mathbb{R}^\chi$ such that the map 
\begin{equation*}
\mathsf{S}: \mathcal{V} \to Q(F) 
\end{equation*}
that assigns to a real $\chi$-tuple $\bar{s} = (s_1,s_2,\ldots, s_\chi)$ the singular-flat surface obtained by shearing $\Sigma_0$ along the horizontal strips by shear parameters $\bar{s}$, is a diffeomorphism to its image. 
\end{prop}

\begin{proof} 

 By the definition of a shear, the new horizontal strips have  periods whose differences with the previous periods are the \textit{real} shear parameters. Note that this operation does not change the \textit{imaginary} parts of the periods. In particular, by (\ref{trans}),  all transverse measures of arcs are unchanged by the shearing operation. Hence,  the shear operation does not change the horizontal measured foliation. 
 
 By Theorem \ref{thm1}, and the discussion at the beginning of this section, the subspace $Q(F)$ is locally homeomorphic to $\mathbb{R}^\chi$, and hence the dimensions of the domain and target match.

The fact that $\mathsf{S}$ is a local diffeomorphism is then a consequence of Theorem 4.12 of \cite{BriSmi}, which asserts that the complex periods across the strips in fact form local parameters for the total space of differentials  $\mathcal{Q}(S, n_1,n_2,\ldots n_k)$.
\end{proof}

\textit{Remark.} It would be interesting to describe a geometric parameterization of the \textit{entire} subspace $Q(F)$ using shear operations on strips. One of the difficulties is that the shearing map $\mathsf{S}$ in Proposition  \ref{shear} , when extended to $\mathbb{R}^{\chi}$, is not proper.  
Indeed, if it were proper, then by the Invariance of Domain the map $\mathsf{S}$ would be a homeomorphism, which cannot hold as $Q(F)$ is not simply-connected: for example, as a bundle over $\mathcal{T}_{g,k}$, the fibers have an $S^1$ factor.  
The non-properness is explained by the phenomenon that as shear parameters diverge, the periods of simple closed curves transverse to the foliation might remain bounded, since any such period is the sum of periods across different strips, which might have opposing signs that allow for cancellations.

\appendix

\section{Solving the partially free boundary problem}

In this appendix we provide the proof of  Proposition \ref{doub}, restated here in a self-contained way. 

We shall follow the arguments in \S4.3 (Doubling trick) in \cite{GW1} almost verbatim. In that paper, the target was a $k$-pronged graph with a \textit{single} vertex; in the present paper, the target is a graph with finitely many vertices, and the arguments are very slightly modified.

For $A$ a conformal annulus, let $\hat{A}$ be the annulus obtained by doubling the annulus $A$ across its boundary component $\partial_-A$.  That is, if we denote, as usual, the boundary components $\partial A = \partial_+ A \sqcup \partial_- A$, then we set  $\hat{A}$ to be the identification space of two copies of $A$, where the two copies of $\partial_- A$ are identified.  We write this symbolically as $\hat{A}= A \sqcup_{\partial_- A} \bar{A}$, where $\bar{A}$ refers to $A$ equipped with its opposite orientation.

\begin{prop}\label{doub2} Let $A$ be a conformal annulus and let $\chi$ be a metric graph with finitely many vertices and edges of finite length.  Fix a continuous map $\phi:\partial_+A\to \chi$ on one boundary component that takes on each value only finitely often, and consider the solution $h:A\to \chi$  to the  partially free boundary problem that requires $h$ to agree with $\phi$ on that boundary component $\partial_+A$. Then this map $h$ extends by symmetry to a solution $\hat{h}$ of the symmetric Dirichlet-problem on the doubled annulus $\hat{A} = A^+ \sqcup_{\partial_{-}A} A^-$ where one requires a candidate $\phi$ to be the map on both boundary components of $\hat{A}$. In particular, we have $h = \hat{h}\vert_{A^+}$. 
\end{prop}

\begin{proof}[Warmup to the proof of Proposition \ref{doub2}] We begin by assuming that the solution $h$ to the partially free boundary problem described above has image $h(\partial_-A)$ of the boundary  component $\partial_-A$ disjoint from the vertices of $\chi$.

By this assumption, near the boundary $\partial_-A$, we have that $h$ is a harmonic map to a smooth (i.e. non-singular) target locally isometric to a segment.

First, we show that for the solution of the partially free-boundary problem,  the normal derivative at the (free) boundary component $\partial_-A$ vanishes. We include the elementary computation below for the sake of completeness. \\

Consider a family of maps $u_t: A \to \R$ defined for $t \in (-\epsilon, \epsilon)$.  A map $u_0=h$ in this family is critical for energy if
\begin{align*}
0 &=\frac{d}{dt}\Bigr|_{t=0} E(u_t) \\
&= \frac{d}{dt}\Bigr|_{t=0} \frac{1}{2}  \iint_A |\nabla u_t|^2 dvol_A \\
&= \iint_A \nabla \dot{u} \cdot \nabla u_0 dvol_A \\
&= -\iint_A \dot{u} \Delta u_0  dvol_A  + \int_{\partial A} \dot{u} \frac{\partial}{\partial \nu} u_0 dvol_{\partial A}
\end{align*}
where the last equality is obtained by an integration by parts. (Here $ \frac{\partial}{\partial \nu}$ indicates the outward normal derivative to the boundary.)\\

Thus, since $\dot{u} = \frac{d}{dt}\Bigr|_{t=0} u_t$ is arbitrary (and vanishes on the boundary component where the value of $u$ is fixed), we see that the
necessary conditions for a solution $u_0$ to the partially free boundary value problem are that 
\begin{align*}
\Delta u_0 &=0  \text{ (in the interior of $A$)}
\end{align*}
\begin{align}\label{normal}
\frac{\partial}{\partial \nu} u_0 &=0 \text{ (on the free boundary).}
\end{align}

We then show that the partially free boundary solution $h$ is ``half" of a Dirichlet problem on a doubled annulus.  We follow an approach developed by A. Huang in his Rice University thesis (see Lemma 3.5 of \cite{Huang16}).

Let $\hat{h}: \hat{A} \to \chi$ denote the map defined on $\hat{A}$ that restricts to $h$ on the inclusion $A \subset \hat{A}$ and, in the natural reflected coordinates, on the inclusion $\bar{A} \subset \hat{A}$. By the continuity of $h$ on $A$ and its closure, it is immediate that $\hat{h}$ is continuous on $\hat{A}$.  The vanishing of the normal derivative at the boundary (\ref{normal})  implies that the gradient $\nabla h|_{\partial_- A}$ is parallel to $\partial_- A$.  As that gradient is continuous on $A$ up to the boundary (see e.g. \cite{Evans}, Theorem~6.3.6), we see that $\hat{h}$ has a continuously defined gradient on the interior of the doubled annulus $\hat{A}$. 

Next, note that because $\hat{h}$ is $C^1$ on $\hat{A}$, we have that $\hat{h}_i$ is weakly harmonic on $\hat{A}$.  In particular, we can invoke classical regularity theory to conclude that $\hat{h}$ is then smooth and harmonic on $\hat{A}$.  Thus, since $\chi$ is an NPC space, the map $\hat{h}$ is the unique solution to the Dirichlet harmonic mapping problem of taking $\hat{A}$ to $\chi$ with boundary values $h|_{\partial_+ A}$.

This concludes the proof of the model case.\\

Next, we adapt this argument to the general case when the image of the boundary $h(\partial_-A)$ might possibly contain a  vertex  of the graph $\chi$.  \\

To accomplish the extension to the singular target case, we first analyze the behavior of the level set $h^{-1}(O)$ of a vertex $O$ within the annulus
$A$, particularly with respect to its interaction with the free boundary $\partial_- A$. 

\begin{lem}
Under the hypotheses above, any connected component of the level set $h^{-1}(O)$ of a vertex $O$ within the annulus
$A$ meets the free boundary $\partial_- A$ in at most a single point.\end{lem}

\begin{proof}
We begin by noting that the proof of the Courant-Lebesgue lemma (see Lemma~\ref{cour}), based on an energy estimate for $h$ on an annulus (see, for example, Lemma 3.2 in \cite{Wolf2}) extends to hold for half-annuli, centered at boundary points of $\partial_- A$. Applying that argument yields  a uniform estimate on the modulus of continuity of the map $h$ on the closure of $A$ only in terms of the total energy of $h$. Thus there is a well-defined continuous extension of the map $h$ to $\partial_- A$. We now study this extension, which we continue to denote by $h$.

First note that there cannot exist an arc $\Gamma \subset A \cap h^{-1}(O)$ in the level set for $O$ in $A$ for which $\Gamma$ meets $\partial_-A$ in both endpoints of $\partial \Gamma$.  If not, then since $A$ is an annulus, 
some component of $A \setminus \Gamma$ is bounded by arcs from $\partial_-A$ and $\Gamma$.  But as $\partial_-A$ is a free boundary, we could then redefine $h$ to map only to the vertex $O$ on that component, lowering the energy.  This then contradicts the assumption that $h$ is an energy minimizer.

Focusing further on the possibilities for the level set $h^{-1}(O)$, we note that by the assumption on the boundary values of $h$ on $\partial_+A$  being achieved only a finitely many times,  the level set $h^{-1}(O)$  can meet $\partial_+A$ in only a finite number of points (in fact the number of them  is also fixed and equal to a number $K$ in subsequent applications, since the boundary map would be a restriction of the collapsing map for a pole of finite order).

Therefore, with these restrictions on the topology of $h^{-1}(O)$ in $A$ in hand, we see that by the argument in the previous two paragraphs, each component of $h^{-1}(O)$ then must either be completely within, or have a segment contained in $\partial_-A$,  or  - the only conclusion we wish to permit - connects a single point of $\partial_-A$ with a preimage of the vertex on $\partial_+A$.

Consider the first case where a component of $h^{-1}(O)$ is completely contained within $\partial_-A$.  A neighborhood $N$ of a point in such a component then has image $h(N)$ entirely within a single prong, so the harmonic map on that neighborhood agrees with a classical (non-constant) harmonic function to an interval. Thus in a neighborhood of the boundary segment, say on a coordinate neighborhood 
$\{\Im (z) =  y \in [0, \delta) \}$, 
the requirements from equation~\eqref{normal} and that $h(0) = O$ and non-constant imply that the harmonic function $h$ to (i) is expressible locally as $\Im (az^k) + O(|z|^{k+1})$ for some $k \ge 1$ and some constant $a \in \C^*$, 
(ii) be real analytic, and (iii) satisfy $\frac{\partial h}{\partial y} = 0$ (where $z = x+iy$).  It is elementary to see that these conditions preclude this segment $h^{-1}(O)$ from being more than a singleton: that  $h^{-1}(O)$ contains a segment defined by $\{y=0, x\in (-\epsilon, \epsilon)\}$ implies that the constant $a$ in condition (i) is real. But then $0=\frac{\partial h}{\partial y} = \Re (akz^{k-1}) +O(|z|^{k})$ also on that segment $\{y=0, x\in (-\epsilon, \epsilon)\}$: thus $a = 0$, and so the map $h$ must be constant, contrary to hypothesis.

The same argument rules out  the case when the level set $h^{-1}(O)$ meets the free boundary $\partial_-A$ in a segment, and that segment is connected by an arc of $h^{-1}(O)$ to $\partial_+A$.  Namely, for this situation, we apply the argument of the previous paragraph to a subsegment of $h^{-1}(O)$ on $\partial_-A$ with a neighborhood whose image meets only an open prong, concluding as above that such a segment on $\partial_-A$ is not possible. 

Thus the intersection of such a component of the level set $h^{-1}(O)$ with the free boundary $\partial_-A$ is only a singleton, as needed.
\end{proof}

\emph{Conclusion of the proof of Proposition~\ref{doub2}:} It remains to consider the case when the image of the boundary $\partial_-A$ by $h$ contains a vertex $O$.  It is straightforward to adapt, as follows, the argument we gave in the warmup for the smooth case to the singular setting.

 Consider a neighborhood of a point on $h^{-1}(O) \cap \partial_-A$.  Doubling the map on that half-disk across the boundary $\partial_-A$ yields a harmonic map from the punctured disk to the graph (defined everywhere except at the isolated point $h^{-1}(O) \cap \partial_-A$).  That harmonic map is smooth on the punctured disk and of finite energy, and hence has a Hopf differential of bounded $L^1$-norm.  The puncture is then a removable singularity for that holomorphic differential, and hence for the harmonic map. 

The extended map $\hat{h}$ is then harmonic on the doubled annulus, and is the (unique) solution to the corresponding Dirichlet problem, as required. 

(Note that  the normal derivative of the map may have a vanishing gradient at the boundary prior to doubling; this results in a zero of the Hopf differential on the central circle of the doubled annulus.)\end{proof}

\section{Proof of Lemma~\ref{decay}} \label{appendix:decay}

Consider first the  special case when $f(\theta) = Me^{in\theta}$ where $n\geq 1$ and $M$ is a real coefficient. 

We compute that the Laplace equation $\triangle h = 0$ has solution
\begin{equation}\label{solution}
h(x,\theta)  = \left( \frac{\sinh nx + \sinh n(L-x) }{\sinh nL} \right) Me^{in\theta} 
\end{equation}
where we have used the boundary conditions $h(0,\cdot) = h(L, \cdot) = f$. 

Thus, at $x=L/2$ we then obtain
\begin{equation}\label{term}
\lvert h(L/2,\theta) \rvert  \leq K \cdot Me^{-\lvert n \rvert L/2}
\end{equation}
for some (universal) constant $K$.

For general boundary values $f$, we have the Fourier expansion
\begin{equation*}
f(\theta) = \sum\limits_{n\neq 0} M_ne^{in\theta}
\end{equation*}
where note that there is no constant term because the mean of the boundary map $f$ vanishes.
The coefficients of $f$ satisfy
\begin{equation}\label{fbd}
\sum\limits_{n\neq 0} \lvert M_n\rvert^2 = \lVert f \rVert_2^2 \leq M^2
\end{equation}
and the general solution is:
\begin{equation*}
h(x,\theta)  =  \sum\limits_{n\neq 0} \left( \frac{\sinh nx + \sinh n(L-x) }{\sinh nL} \right) M_ne^{in\theta} 
\end{equation*}

From (\ref{term}) we find:

\begin{equation}\label{bd1}
\lvert h(L/2,\theta) \rvert  \leq   \sum\limits_{n\neq 0} K\cdot M_ne^{-\lvert n\rvert L/2}
\end{equation}

\medskip 
Note that the geometric series
\begin{equation}\label{geom}
  \sum\limits_{n=1}^\infty e^{-nL} =  \left( \frac{e^{-L}}{1 - e^{-L}} \right)  \leq  (K^\prime)^2e^{-L}
\end{equation}
for the constant $K^\prime = (1-e^{-1})^{-1/2} \approx 1.26$ (once we assume that $L>1$).

By the Cauchy-Schwarz inequality on (\ref{bd1}) and using (\ref{fbd}) and (\ref{geom}) , we then get:
\begin{equation*}
\lvert h(L/2,\theta) \rvert  \leq   2K\cdot M \cdot K^\prime e^{-L/2}
\end{equation*}
which is the required bound.\\

For proving the decay of derivative along the $\theta$-direction, note that from (\ref{solution}) we have
\begin{equation*}
h_\theta (x,\theta)  =  \sum\limits_{n\neq 0} \left( \frac{\sinh nx + \sinh n(L-x) }{\sinh nL} \right) n\cdot M_ne^{in\theta} 
\end{equation*}

and thus just as in (\ref{term}) we can derive:

\begin{equation}\label{bd1}
\lvert h_\theta(L/2,\theta) \rvert  \leq   \sum\limits_{n\neq 0} \lvert n\rvert \cdot  K\cdot M_ne^{-\lvert n\rvert L/2}
\end{equation}

\medskip 
This time estimating the series yields
\begin{equation}\label{geom}
  \sum\limits_{n=1}^\infty n \cdot e^{-nL} =   \frac{e^{-L}}{(1 - e^{-L})^2}  \leq  (K^{\prime\prime})^2 e^{-L}
\end{equation}
for the constant $K^{\prime\prime} = (1-e^{-1}) \approx 1.6$ (once we assume that $L>1$). 

Thus, by an application of the Cauchy-Schwarz inequality as before, we get:
\begin{equation*}
\lvert h_\theta(L/2,\theta) \rvert  \leq   2K\cdot M \cdot K^{\prime\prime} e^{-L/2}
\end{equation*}
which is the analogous exponential decay for the $\theta$-derivative of $h$.\\
\qed

\bibliographystyle{amsalpha}
\bibliography{qdref3}

\end{document}